\documentclass{amsart}
\usepackage{amssymb}
\usepackage{amsmath}
\usepackage{mathrsfs}
\usepackage{amsfonts}
\usepackage{amscd}
\usepackage{pifont}
\usepackage{txfonts}
\usepackage{dsfont}

\newtheorem{theorem}{Theorem}[section]
\newtheorem{proposition}[theorem]{Proposition}
\newtheorem{lemma}[theorem]{Lemma}

\newcommand\Prefix[3]{\vphantom{#3}#1#2#3}

 \theoremstyle{definition}
 
 \newtheorem{remark}[theorem]{Remark}
  \theoremstyle{remark}
\numberwithin{equation}{section}
\begin{document}

\title[Degenerate Hessian Equations]
{Existence and convexity of local solutions\\
to degenerate Hessian equations}

\author{Guji Tian and Chao-Jiang Xu}
\address{Guji Tian,
Wuhan Institute of Physics and Mathematics
\newline\indent
Chinese Academy of Sciences,Wuhan, P.R. China, 430071}
\email{tianguji@wipm.ac.cn}

\address{Chao-Jiang Xu,
Universit\'e de Rouen, CNRS UMR 6085, Laboratoire de Math\'ematiques
\newline\indent
76801 Saint-Etienne du Rouvray, France
\newline\indent
and
\newline\indent
School of Mathematics and statistics, Wuhan University  430072,
Wuhan, P. R. China
}
\email{chao-jiang.xu@univ-rouen.fr}

\date{\today}



\maketitle

\begin{abstract}
In this work, we prove the  existence of local convex solution
to  the following $k-$Hessian equation
$$
S_{k}[u]=K(y)g(y, u, Du)
$$
in the neighborhood of a point $(y_0,u_0,p_0)\in \mathbb
R^{n}\times\mathbb R\times \mathbb R^{n}$, where $g\in C^\infty, g(y_0,u_0,p_0)>0,$  $K\in C^\infty$ is nonnegative near $y_0$, $K(y_0)=0$ and $\mbox{Rank}\, (D^2_y K) (y_0)\geq n-k+1.$
\end{abstract}



\section{Introduction}\label{section1}

In this work, we study the following
$k$-Hessian equation :
\begin{equation}\label{eq:1.1}
S_{k}[u]=f(y, u, Du),
\end{equation}
on the open domain $\Omega\subset\mathbb{R}^{n}$ with $2\le k\le n$, where  $f\ge 0$ is defined on $\Omega\times \mathbb{R}\times \mathbb{R}^n$ with $f(y_0,u_0,p_0)=0$.  When $u\in C^2$, the $k$-Hessian operator
$S_{k}[u]$ is defined by
$$
S_{k}[u]=S_k(D^2u)=\sigma_{k}[\lambda(D^{2}u)]=\sum_{1\leq i_{1}<
i_{2}\ldots<i_{k}\leq
n}\lambda_{i_{1}}\lambda_{i_2}\ldots\lambda_{i_{k}},
$$
where $S_k(D^2u)$ is the sum of all $k$-order principal minors of the Hessian matrix $ (D^2u)$, and $\lambda(D^{2}u)=(\lambda_{1}(D^{2}u),\ldots,\lambda_{n}(D^{2}u))$ are the
eigenvalues of the matrix $(D^{2}u)$. One origin of $k-$Hessian operator is from Christoffel-Minkowski problem, see \cite{GG,GM,GLM,GMZ}
and references therein, another one is from calibrated geometries in \cite{HL}. The background of $k-$Hessian operator in terms of differential geometry   can also be found in Section 4, \cite{IPY}.

When $f>0$, the solutions $u$ of \eqref{eq:1.1} is considered with $\lambda(D^2u)$ in the so-called G{\aa}rding cone :
\begin{equation*}
\Gamma_{k}(n)=\{\lambda=(\lambda_1,\lambda_2,\ldots,\lambda_n)\in\mathbb{R}^{n};\,
\sigma_j(\lambda)>0,1\leq j\leq k\}.
\end{equation*}
If $f\ge 0$,  the equation \eqref{eq:1.1} is called degenerate, in this case, we consider the solutions with
$$
\lambda(D^2u)\in \bar\Gamma_{k}(n)=\{\lambda\in \mathbb{R}^{n};\, \sigma_{j}(\lambda)\geq0
,1\leq j\leq k\}.
$$
A function $u\in C^2$ is called to be $k$-convex, if $\lambda(D^{2}u)\in \bar\Gamma_{k}(n). $ The $n$-convex function is simply called convex.

The convexity of solutions of \eqref{eq:1.1}  is an important problem in the field of
geometries analysis, including usual convexity,  power convexity, log-convexity, or quasi-convexity. For example, in the study of Christoffel-Minkowski problem (see \cite{GG,GM,GLM,GMZ}), an important subject is to prove the existence of a convex
body with prescribed area measure of suitable order, this is equivalent  to prove the microscopic convexity principle (constant rank theorem) for some $k$-Hessian type equation on the unit sphere $\mathbb{S}^n$. There is also a strong connection between convexity properties of solutions to elliptic and parabolic partial differential equations and Brunn-Minkowski type inequalities for associated variational functionals, see \cite{LMX,MX,SP}. When Guan \cite{GB} uses subsolution in place of all curvature restrictions on $\partial\Omega$ to construct local barriers for boundary estimates, it is an assumption that there exists a locally strictly convex function in $C^2(\overline{\Omega}),$ see Theorem 1.1 and 1.3, \cite{GB}.

The  microscopic convexity principle, with applications in geometric equations on manifolds, has been established in \cite{BG} for the very general fully nonlinear elliptic and parabolic operators of second order. Guan, Spruck and Xiao \cite{GSX} point out that the asymptotic Plateau problem for finding a complete strictly locally convex  hypersurface is reduced to the Dirichlet problem for a fully nonlinear equation, a special form of which is the $k-$Hessian equation, see Corollary 1.11, \cite{GSX} and they have proved the existence of  such hypersurface, it is especially interesting that they have proved that,  if $\partial\Omega$ is strictly (Euclidean) star-shaped about the origin, so is the unique solution,  see Theorem 1.5,  \cite{GSX}.
For the $k$-Hessian equation with  $k=2, n=3$, the power convexity for Dirichlet problem of equation \eqref{eq:1.1} with $f=1$, and log-convexity for the eigenvalue problem have been studied in \cite{LMX,MX}, see also \cite{SP}. The above convexity results are established on two facts, one is that the equations are elliptic, another is that the existence of classical  (at least $C^2$) solution has already known or can be proved. However, in many important geometric problem, the associated $k$-Hessian equation is degenerate (see \cite{HH}), and for the degenerate elliptic $k$-Hessian equation, one can only prove the existence of $C^{1,1}$ solution for Dirichlet problem (\cite{ITW}).

In this paper, we study the convexity of solution with following definition: for a convex domain $E$, the function $v\in C(E)$  is said to be {\em strictly convex} if
$$
v(ty+(1-t)z)< t \, v(y)+(1-t) v(z),\,\, 0<t <1,\,\, y,\, z\in E,\,\, y\neq z.
$$
Which, in case $v\in C^2(E)$, is equivalent to
\begin{equation}\label{eq:1.3}
 \sum_{i,j=1}^n(y_i-z_i)(y_j-z_j) \int_0^1\int_0^1\frac{\partial^2v}{\partial {x_i}\partial x_j}(x(s,\mu))\, d\mu\, ds>0
\end{equation}
with $x(s,\mu)=(s\mu +(1-s)t)y+(s(1-\mu)+(1-s)(1-t))z$, this shows that the positive definiteness of Hessian matrix $(D^2v)$ is a sufficient condition for
the strict convexity, but not necessary.

However if $u\in C^2$ is a $k$-convex solution of $S_k[u] = f (y)\ge 0$ with $2\le k < n$
and $f(y_0) = 0$, then $S_{k+1}[u](y_0) > 0$ will never occur, so that there are two possibilities: 1)
$S_{k+1}[u](y_0) < 0$, in this case, $u$ is not $(k + 1)$-convex; 2) $S_{k+1}[u](y_0) = 0$, in this case,  it is shown in
Theorem 1.1 of \cite{TWX2} that, if $S_k[u](y_0) = S_{k+1}[u](y_0) = 0$, then $S_l[u](y_0) =0$ for $k\leq l\leq n.$
 In particular, if $f\equiv 0$, since the case $S_{k+1}[u]>0$ will never occur, then either $S_{k+1}[u]<0$ in some open subset of $\Omega$ or $S_{k+1}[u]\equiv0$  in $\Omega$ itself,  in the former case  $u$ is not $(k+1)-$convex and let alone strictly convex; in the latter case, $S_{l}[u]\equiv 0$ for $k\leq l\leq n$, then  the graph $(y; u(y))$ for $k = 2$
has the vanishing sectional curvature and  must be a cylinder or a plane, see \cite{Shi} and \cite{Ush},
meanwhile the graph $(y; u(y))$ for $k > 2$, by Lemma 3.1 of \cite{Har}, is a surface of constant
nullity (at least) $n-k + 1$ and then is a $(n-k + 1)$-ruled surface. Therefore if  $f\equiv 0$, the
solution $u$ to \eqref{eq:1.1} is not strictly convex (at least) along the rulings.

Motivated by above analysis,  in this work, we study the local solutions for the following equation, $2\le k\leq n$,
\begin{equation}\label{eq:1.4}
S_k[u]=K(y)g(y,u,Du)\, ,
\end{equation}
with the following assumptions
\begin{equation*}
(H)\qquad\qquad \left\{
\begin{array}{l}
 K\in C^\infty  \ \mbox {is nonnegative in a neighbourhood of} \ y_0\in \mathbb{R}^n, \\
K(y_0)=0, \mbox{ Rank}\, (D^2 K)(y_0)\geq n-k+1 ,\\
 g\in C^\infty \  \mbox{near} \  Z_0=(y_0,u_0,p_0)\, \mbox{and} \, g(Z_0)>0.
\end{array}\right.
\end{equation*}
This assumption is independent of coordinates. Our main Theorem is :

\begin{theorem}\label{main}
If $K, g$ satisfy the assumption (H), then for any $s\ge  2[\frac n 2]+5$, the equation \eqref{eq:1.4} admits a  strictly convex $H^s$-local solution in a neighbourhood of  $y_0\in \mathbb{R}^n$.
\end{theorem}

Remark that $u=\frac{1}{2}\sum_{i=1}^{n-1}y_i^2+\frac{1}{12}y_n^4$ is a strictly convex solution of the following Monge-Amp\`{e}re equation :
$$
\det D^2u= y_n^2\, .
$$
But the Hessian matrix $(D^2 u)$ is not positive definite at origin.

This article is arranged as follow: In Section \ref{section2a}, we will introduce the idea of how to construct convex local
solution in terms of $K(y)$. In Section \ref{section2}, we will construct the first order approximate solution $\psi(y)$
which is strictly convex. The Section \ref{section3} will be devoted to proving the degenerate ellipticity of the linearized operator
of $S_k[u]$ around $\psi$.  In Section \ref{section4}, we will use Lax-Milgram theorem to prove the existence
of $H^0$-weak solution and their a priori $H^s$-estimates. In Section \ref{section5}, we will prove the existence
of $k$-convex solution by the Nash-Moser-H\"ormander iteration procedure. In section \ref{section6}, we will prove the strict convexity of
the local solution obtained in Section \ref{section5}. {Section \ref{section7}, as an appendix, is devoted to  the estimates of eigenvalues and eigenvectors for a matrix after a  small perturbation, the conclusion of which is a generalization of Lemma 1.1 of \cite{HZ}.


\section{Schema of construction of convex local solutions} \label{section2a}

The assumption (H) is independent of coordinates. Now we choose special coordinates under which  the leader term of  the solution can be explicitly expressed. Since the degeneracy is come from the term $K$, for the simplicity of notations and also
 computation, we suppose that
$$
(B)\qquad\qquad g\equiv 1,\,\,  \mbox{Rank}\, (D^2 K)(y_0)= n-k+1 .
$$
By a translation $y\rightarrow y+y_0$ and  a change of unknown function $u\rightarrow u-u(y_0)-Du(y_0)\cdot y,$ we can assume $Z_0=(0, 0, 0)$. On the other hand, the solution is searched for  locally, that is, we assume without loss of generality that $K(y)$ is defined in some neighborhood of origin and
\begin{equation}\label{eq:8.2}
K(y)=\sum_{j=k}^n c_jy_j^2+O(|y|^3)
\end{equation}
where  $2c_j>0, \, k\leq j\leq n$ are the positive eigenvalues of $(D^2K)(0)$.
In order to describe the ``localness'',  small $\varepsilon>0$  is introduced by the change of variables $y=\varepsilon^2x$,
we fix a domain
$$
 \Omega=Q_\pi\times Q_{\delta_0}\subset \mathbb{R}^{k-1}\times \mathbb{R}^{n-k+1}
$$
with $\delta_{0}>0$ being chosen small later and $x'=(x_1,\cdots,x_{k-1}), x''=(x_k,\cdots,x_n), $
$$
Q_\pi=\{x'\in \mathbb{R}^{k-1};\,\, |x_{i}|< \pi,1\leq i\leq k-1\},
\quad
Q_{\delta_0}=\{ x''\in \mathbb{R}^{n-k+1};\,\, \sum_{i=k}^n|x_i|^2< \delta_0^2\}.
$$
We will determine some $\varepsilon_0>0$ and study the equation \eqref{eq:1.4} in the following form
 \begin{equation}\label{eq:8.8++}
S_k[u]=\tilde{K} \ \ \mbox{in}\ \ \Omega_{\varepsilon_0}=\left\{y=\varepsilon_0^2x; x\in \Omega\right\}
\end{equation}
with
\begin{equation}\label{eq:8.14}
 \tilde{K}(y)=(1- \chi(\varepsilon^{-2}y'))\sum_{i=k}^{n}c_iy_i^2+\chi(\varepsilon^{-2}y')K(y),
\end{equation}
where $\chi (x')\in C^\infty_0(Q_{\pi})$ is a cutoff function equal to $1$ if $|x'|\leq \frac{\pi}{2}$, equal to zero if $|x'|\geq \pi,$ and $0\leq \chi\leq 1$ . The local solution of \eqref{eq:8.14} is also the one of \eqref{eq:1.4}.  The aim of introduce of function $\chi(x')$ is to guarantee  the periodicity with respect to variable $x'$ for nonhomogeneous terms and the coefficients  of  all the linearized equations, which is important and convenient for existence of solution because the linearized operator $L_G(w)$ of \eqref{eq:8.18} may be degenerate in the direction $x'.$

We will construct the local solution of equation \eqref{eq:8.8++} in the following form
\begin{equation}\label{eq:2.abc+}
u(y)=\frac{1}{2}\sum_{j=1}^{k-1}\tau_jy_j^2+P(y)+\varepsilon^{\frac{17}{2}} w(\varepsilon^{-2}y).
\end{equation}
So the construction of solution is by three steps:

\smallskip
{\bf 1) Solutions of the homogeneous equation}

Let $\tau=(\tau_1, \cdots, \tau_{k-1}, 0,\cdots, 0)$ with $\tau_1>\cdots> \tau_{k-1}>0$, then the convex function
\begin{equation}\label{2.5}
\varphi(y)=\frac{1}{2}\sum_{j=1}^{k-1}\tau_jy_j^2
\end{equation}
satisfies the homogenous equation $S_k[\varphi]=0$, and the linearized operators
\begin{equation}\label{2.5+1}
\mathcal{L}_\varphi=\sum_{j=1}^n \sigma_{k-1, j}(\tau)\partial^2_j
\end{equation}
is degenerate elliptic with
$$
 \sigma_{k-1, j}(\tau)=0, 1\le j\le k-1; \,\,\, \sigma_{k-1, j}(\tau)= \sigma_{k-1}(\tau)=\prod^{k-1}_{l=1}\tau_l>0,\,\, k\le j\le n.
$$
We have also  $S_{k+1}[\varphi]=\cdots=S_n[\varphi]=0$.

Remark that, in \cite{TWX,TWX2}, we choose $\tau\in \partial\Gamma_k(n)$  with  $\sigma_{k+1}(\tau)<0$,
so $\varphi$ is not $(k+1)$-convex; also the solutions constructed in \cite{CH} must not be  $(k+1)$-convex, because in these cases every linearized operator \eqref{2.5+1} is uniformly elliptic. But in present  work, we want to construct the local strictly convex solution, so we can't make that choice. On the other hand, the function $\varphi$
defined in \eqref{2.5} is only {\em weakly} convex, so  it is difficult to guarantee the convexity after a perturbation.

\smallskip
{\bf 2) Approximate strictly convex solution}

Using the assumption (H) on $K$, we construct a function $P$ such that
\begin{equation}\label{2.6}
\psi(y)=\frac{1}{2}\sum_{j=1}^{k-1}\tau_jy_j^2+P(y)
\end{equation}
satisfies
$$
S_k[\psi]=\tilde K+O(1)\varepsilon^{\frac 92},
$$
and $\psi$ is strictly convex on $\Omega_{\varepsilon_0}$. The construction of the function $P$ is algebraic by using
the assumption (H) of $K$.

\smallskip
{\bf 3) Nash-Moser-H\"ormander iteration}

We construct finally the smooth function $w$ such that the function $u$ defined by \eqref{eq:2.abc+} is a local solution
of  equation \eqref{eq:8.8++}. We use the Nash-Moser-H\"ormander iteration procedure:
\begin{equation}\label{eq:6.9+1}
\left\{
\begin{array}{l}
 w_0=0,w_{m+1}=w_m+S_m\rho_m\\
L_G(w_m)\rho_m+\theta_m\triangle \rho_m=g_m, \indent \text{in} \indent x\in \Omega.
\end{array}\right.
\end{equation}
where $\{S_m\}$ is a family of smoothing operators,
\begin{equation}\label{eq:6.10}
 g_m=-G(w_m)=\frac{1}{\varepsilon^{\frac{9}{2}}}\left\{S_k(\psi+\varepsilon^{\frac{17}{2}} w_m(\varepsilon^{-2}y) )- \tilde{K}\right\},\,\, \theta_m=\sup_{\Omega}|G(w_m)|=\|g_m\|_{L^\infty},
\end{equation}
and
\begin{equation*}
L_{G}(w)=\sum_{i,j=1}^{n}\frac{\partial S_{k}(\mathbf{r}(w))}{\partial r_{ij}}\partial_{i}\partial_{j},
\end{equation*}
where
$$
\mathbf{r}(w)=\left(\sum_{l=1}^{k-1}\delta_{i}^{j}\delta_j^l\tau_l+ P_{ij}(\varepsilon^2x)+
 \varepsilon^{\frac{9}{2}} w_{ij}(x)\right)_{1\le i, j\le n}.	
$$
The procedure is to prove the existence and the convergence of the sequence $w_m\,\to\, w$ in some Sobolev space with
$\theta_m\to 0$ which imply that the function $u$ defined in \eqref{eq:2.abc+} by $w$ is a local solution of equation \eqref{eq:8.8++}. We also need to prove that
the perturbation doesn't destruct the strictly convexity of $\psi$ constructed by \eqref{2.6}.

Remark that the linearized equation is degenerate elliptic, so that there is a loss of the regularity for the \`a priori estimate of solution
$\rho_m$,  but the coefficients of linearized operators depends on $D^2 w_m$, so that we need to smoothing the solution $\rho_m$
to continue the iteration \eqref{eq:6.9+1} for $m\in\mathbb{N}$. This is quite different from the procedure of iteration used in \cite{TWX2} where the linearized equation is uniformly elliptic.


\section{The first order approximate solutions} \label{section2}

Since $K(y)$ attains its minimum $0$ at origin, the critical-point theorem implies $\nabla K(0)=0$.
Then we have

\begin{proposition}
Suppose that $K(y)$ satisfies assumption $(H)$ and \eqref{eq:8.2}, then we have the following decomposition
$$
K(y)= K(y',0)+\frac{1}{2}\sum_{i=k}^n \frac{\partial^2K}{\partial y_i^2}(y',0)y_i^2+R(y)
$$
where $K(y',0)$  vanishes at $y'=0$ up to  order greater than four and
$$
R(y)=\sum_{i=k}^n\frac{\partial K}{\partial y_i}(y',0)y_i+\frac{1}{2}\sum_{i,j=k,i\neq j}^n\frac{\partial^2K}{\partial y_i\partial y_j}(y',0)y_iy_j+O(1)|y''|^3.
$$
\end{proposition}
In particulary, for $y=\varepsilon^2 x, x\in\Omega$,  we have
$$
R(y)=O(1)\varepsilon^{6}.
$$

Formally, if $u\in C^{3,1}$ is a solution to \eqref{eq:1.4}, by Taylor expansion
\begin{equation}\label{eq:1.5}
u(y)=\sum_{i,j=1}^nu_{ij}(0)y_iy_j+o(|y|^2)
\end{equation}
substituting \eqref{eq:1.5} into \eqref{eq:1.4}, we see that,
$$
S_k(D^2u(0))=S_k(u_{ij}(0))=\lim_{y\rightarrow 0}S_k(D^2u(y))=K(0)=0.
$$
Therefore, a smooth (at least $C^{3,1}$) local solution to \eqref{eq:1.4} is a solution of $S_k[u]=0$  plus a small perturbation $o(|y|^2)$.
So we wish to construct the  first order approximate solution as following form
$$
\psi(y)=\frac{1}{2}\sum_{i=1}^{k-1}\tau_iy_i^2+P(y)
$$
with $\tau_1>\tau_2>\ldots>\tau_{k-1}>0$, such that
$$
S_k[\psi]=\tilde{K}+O(1)\varepsilon^{\frac 92},
$$
which is difficult to arrive at. Our observation is that $\sigma_{k-1}(\tau)\sum_{j=k}^n P_{jj}(y)$ is the main part of $S_k[\psi]$, so we only  need to find out $P(y)$ to satisfy the weaker version
$$
\sigma_{k-1}(\tau)\sum_{j=k}^n P_{jj}(y)=\tilde{K}-\chi(\varepsilon^{-2}y')R(y).
$$
This is our trick how to construct $P(y).$ Let
\begin{equation}\label{eq:8.8}
\begin{split}
P(y)&\equiv \frac{1}{2(n-k+1)\sigma_{k-1}(\tau)}\chi(\varepsilon^{-2}y')K(y', 0)\sum_{j=k}^ny_j^2\\
&+\frac{1}{24\sigma_{k-1}(\tau)}\chi(\varepsilon^{-2}y')\sum_{i=k}^n\left(\frac{\partial^2K}{\partial y_i^2}(y',0)-2c_i\right)y_i^4
\\
&+\frac{1}{12}\sum_{i=k}^n\left[\frac{c_i}{\sigma_{k-1}(\tau)}-4\alpha (n-k)\right]y_i^4+\alpha\sum_{j=k}^n\sum_{i=k,i\neq j}^ny_i^2y_j^2
\end{split}
\end{equation}
where
\begin{equation}\label{eq:1.15++}
0<\alpha<\frac{1}{16(n-k)^2+4(n-k+1)} \min_{k\leq j\leq n}\left\{\frac{c_j}{2\sigma_{k-1}(\tau)}\right\}.
\end{equation}

\begin{remark}
Giving an example, $K^1(y)=K^1(y'')=\sum_{i=k}^{n}c_iy_i^2$, then
$$
P(y)=P^1(y'')=\frac{1}{12}\sum_{i=k}^n\left[\frac{c_i}{\sigma_{k-1}(\tau)}-4\alpha (n-k)\right]y_i^4+
\alpha\sum_{j=k}^n\sum_{i=k,i\neq j}^ny_i^2y_j^2.
$$
We have
$$
\sigma_{k-1}(\tau)\sum_{j=k}^n P^1_{jj}(y'')=K^1(y''),
$$
and the strictly convex function $\psi^1(y)=\frac{1}{2}\sum_{i=1}^{k-1}\tau_iy_i^2+P^1(y'')$ satisfies
$$
S_k(\psi^1)=K^1(y'')+O(|y''|^4).
$$
\end{remark}

The direct calculation gives

\begin{proposition}
Let $P(y)$ be defined in \eqref{eq:8.8}, then
\begin{equation}\label{eq:8.13}\left\{
\begin{array}{l}
\sigma_{k-1}(\tau)\sum_{j=k}^n P_{jj}(y)=\tilde{K}-\chi(\varepsilon^{-2}y')R(y);\\
 P_{jj}(y',0)=O(1)|y'|^4,\quad k\le j\le n,\\
 P_{ij}(y)=8\alpha y_iy_j,\ \ \  i\neq j, \quad k\le i,j\le n,
\end{array}\right.
\end{equation}
and
\begin{equation}\label{eq:8.13+}
 P_{jj}(y)\geq 4\alpha |y''|^2,\quad k\le j\le n.
\end{equation}
For small $|y|$,  the minor matrix $(P_{ij})_{k\leq i,j\leq n}$ is  strictly diagonally dominant, more explicitly,  for fixed $k\leq j_0\leq n,$
\begin{equation}\label{eq:8.14++}
 P_{j_0j_0}(y)\geq 2\alpha |y''|^2+\sum_{i=k,i\neq j_0}^n|P_{ij_0}(y)|.
\end{equation}
\end{proposition}

\begin {proof}
From \eqref{eq:8.8}, we obtain, for fixed $k\leq j\leq n$,
\begin{equation}\label{eq:8.9}
\begin{split}
P_{jj}(y)&= \frac{1}{(n-k+1)\sigma_{k-1}(\tau)}\chi(\varepsilon^{-2}y')K(y', 0)\\
&+\frac{1}{\sigma_{k-1}(\tau)}\chi(\varepsilon^{-2}y') {\left[\frac{1}{2}\frac{\partial^2K}{\partial y_j^2}(y',0)-c_j\right]}y_j^2\\
&
+\left[\frac{c_j }{\sigma_{k-1}(\tau)}-4\alpha (n-k)\right]y_j^2
+4\alpha \sum_{i=k,i\neq j}^ny_i^2.\\
\end{split}
\end{equation}
By \eqref{eq:8.9} we have
\begin{equation*}
\sigma_{k-1}(\tau)\sum_{j=k}^n P_{jj}(y)=\chi(\varepsilon^{-2}y')\left( K(y',0)+\frac{1}{2}\sum_{i=k}^n \frac{\partial^2K}{\partial y_i^2}(y',0)y_i^2\right)+ (1-\chi(\varepsilon^{-2}y'))\sum_{j=k}^nc_jy_j^2
\end{equation*}
which proves the first equality in \eqref{eq:8.13}. Since $K(y',0)$  vanishes up to
 order greater than four, we have
$$
P_{jj}(y',0)= \frac{1}{(n-k+1)\sigma_{k-1}(\tau)}\chi(\varepsilon^{-2}y')K(y',0)=O(1)|y'|^4,
$$
then  the second equality in \eqref{eq:8.13} is true. The third equality in \eqref{eq:8.13} is obvious.

Now we return to \eqref{eq:8.9} for $P_{jj}(y)$. Since $ K(y', 0)\geq 0$,
employing \eqref{eq:8.2}, choosing small $|y'|$ and then taking $\alpha $ to satisfy \eqref{eq:1.15++}, we have
$$
P_{jj}(y)\geq \left[\frac{c_j}{2\sigma_{k-1}(\tau)}-4\alpha (n-k)\right]y_j^2+4\alpha \sum_{i=k,i\neq j}^ny_i^2
$$
which implies \eqref{eq:8.13+}. By virtue of   inequality above , for fixed $j_0$ with  $k\leq j_0\leq n,$
we obtain by Cauchy inequality,
\begin{equation*}
\begin{split}
P_{j_0j_0}(y)&\geq [\frac{c_{j_0}}{2\sigma_{k-1}(\tau)}
-4\alpha(n-k)]y_{j_0}^2+4\alpha\sum_{i=k,i\neq j_0}^ny_i^2 \\
\geq& 2\alpha |y''|^2+ \frac{1}{n-k}\sum_{i=k,i\neq j_0}^n\left[\Big[\frac{c_{j_0}}{2\sigma_{k-1}(\tau)}
-4\alpha (n-k+1)\Big]y_{j_0}^2+2\alpha y_i^2\right]\\
\geq& 2\alpha |y''|^2+ \sum_{i=k,i\neq j_0}^n\frac{2}{n-k}|y_iy_{j_0}|\sqrt{2\alpha \Big[\frac{c_{j_0}}{2\sigma_{k-1}(\tau)}
-4\alpha (n-k+1)\Big]}\\
\geq & 2\alpha |y''|^2+ \sum_{i=k,i\neq j_0}^n8\alpha |y_iy_{j_0}|=2\alpha |y''|^2+\sum_{i=k,i\neq j_0}^n|P_{ij_0}(y)|,
\end{split}
\end{equation*}
so the minor matrix $(P_{ij})_{k\leq i,j\leq n}$ is  strictly diagonally dominant and \eqref{eq:8.14++} is proved.
\end {proof}

We will construct the solution as a perturbation of the strictly convex function $\psi(y)$ in the following form,
\begin{equation}\label{eq:2.abc}
u(y)=\psi+\varepsilon^{\frac{17}{2}} w(\varepsilon^{-2}y)=\frac{1}{2}\sum_{j=1}^{k-1}\tau_jy_j^2+P(y)+\varepsilon^{\frac{17}{2}} w(\varepsilon^{-2}y),
\end{equation}
see \eqref{eq:2.abc+}, where $w(x)$ will be proved to be a smooth function later.

By a  change of variable $x=\varepsilon^{-2}y$, the Hessian matrix of $u$ defined in \eqref{eq:2.abc} is
\begin{equation}\label{2.abd}
(D^2 u)(\varepsilon^{2}x)=\mathbf{r}=\left(\sum_{l=1}^{k-1}\delta_{i}^{j}\delta_j^l\tau_l+ P_{ij}(\varepsilon^2x)+
 \varepsilon^{\frac{9}{2}} w_{ij}(x)\right).
\end{equation}
We study firstly the minor matrix $\mathbf{r}_{k-1}=(r_{ij})_{1\leq i,j \leq k-1}$ which is real-valued and symmetric,
then there is an orthogonal  $(k-1)\times (k-1)$ matrix $\widetilde{Q}$ such that
$$
\widetilde{Q}(x,\varepsilon)\, \mathbf{r}_{k-1}\, \Prefix ^{t}{\widetilde{Q}(x,\varepsilon)}=
\text{diag}(\lambda_1(x,\varepsilon),\ldots,\lambda_{k-1}(x,\varepsilon)).
$$
Let
\begin{gather*}Q=
\begin{pmatrix}
\widetilde{Q}&0\\
0&\mathbf{I}_{n-k+1,n-k+1}
\end{pmatrix},
\end{gather*}
then
\begin{gather*}
Q\, \mathbf{r}\, \Prefix^tQ=
\begin{pmatrix}
\lambda_1&0&\ldots&0&r_{1,k}&\ldots& r_{1n}\\
0&\lambda_2&\ldots &0&r_{2,k}&\ldots &r_{2n}\\
\vdots &\vdots &\vdots &\vdots &\vdots &\vdots &\vdots\\
0& 0&\ldots& \lambda_{k-1}&r_{k-1,k}&\ldots&r_{k-1,n}\\
r_{k,1}& r_{k,2}&\ldots& r_{k,k-1}&r_{k,k}&\ldots&r_{k,n}\\
\vdots&\vdots&\vdots&\vdots&\vdots&\vdots&\vdots\\
r_{n,1}&r_{n,2}&\ldots&r_{n,k-1}&r_{n,k}&\ldots&r_{n,n}
\end{pmatrix},
\end{gather*}
where all terms $r_{ij}$ in the above matrix are in the form
\begin{equation}\label{2.12}
r_{i j}(x)= P_{ij}(\varepsilon^2x)+
 \varepsilon^{\frac{9}{2}} w_{ij}(x).
\end{equation}
Noticing the existence of ${\chi(\varepsilon^{-2}y')}$ in $P(y)$, we have
$$
P_{ij}(\varepsilon^2x',\varepsilon^2x'')=O(1)\varepsilon^8|x''|^2+O(1)\varepsilon^4|x''|^4, \quad 1\leq i,j\leq k-1,
$$
applying Lemma \ref{lm:matrix} to the minor matrix ${\bf r}_{k-1},$ we obtain that
\begin{equation}\label{eq:8.22}
\lambda_i=\tau_i+O(1)\varepsilon^4,\indent \lambda_1>\lambda_2>\ldots\lambda_{k-1}>0,
\end{equation}
for $\varepsilon>0$ small enough.

\begin{lemma}\label{th:6.4}
Let $\mathbf{r}$  be defined in \eqref{2.abd}, then $S_k(\mathbf{r})=\text{\ding {172}}+\text{\ding {173}}+\text{\ding {174}}$ with
$$
\left\{
\begin{array}{l}
\text{\ding {172}}= \sigma_{k-1}(\lambda_1,\ldots,\lambda_{k-1})\sigma_1(r_{kk},\ldots,r_{nn})-
\sum_{i=k}^n\sum_{j=1}^{k-1}\sigma_{k-1,j}(\lambda_1,\ldots,\lambda_{k-1})r_{ji}^2\\
\text{\ding {173}}=\sigma_{k-2}(\lambda_1,\ldots,\lambda_{k-1})\left[\sigma_2(r_{kk},\ldots,r_{nn})-
\sum_{i=k}^n\sum_{s=k,s\neq i}^n r_{si}^2\right]+O(\sum_{m\geq k;i\geq 1}|r_{mi}|^3)\\
\text{\ding {174}}=\sum_{j=3}^{k-1}\sigma_{k-j}(\lambda_1,\ldots,\lambda_{k-1})O(\sum_{m\geq k;i\geq 1}|r_{mi}|^3)+O(\sum_{m\geq k;i\geq 1}|r_{mi}|^3),
\end{array}\right.
$$
where $r_{mi}=r_{im}=O(1)\varepsilon^4$ for $m\geq k;i\geq 1.$
\end{lemma}

\begin{proof}
Since $S_k(\mathbf{r})$ is invariant under orthogonal transform, then $S_k(\mathbf{r})=S_k(Q \mathbf{r}\Prefix^tQ).$
By virtue of $\eqref{eq:8.22}$, we can separate $S_k(Q\mathbf{r}\Prefix^tQ)$,
which is the sum of all principal minors of order $k$ of the Hessian $(Q\mathbf{r}\Prefix^tQ)$, into three parts : $\text{\ding {172}}$ are the minors
containing $\sigma_{k-1}(\lambda_1,\ldots,\lambda_{k-1})=\prod_{i=1}^{k-1}\lambda_i$; $\text{\ding {173}}$ are the ones containing all of the  elementary symmetric  polynomials of $(k-2)-$order in $\lambda_1,\cdots,\lambda_{k-1}$; $\text{\ding {174}}$ are the ones containing all of the  elementary symmetric  polynomials of order $\leq k-3$ in $\lambda_1,\ldots,\lambda_{k-1}$.
\end{proof}

\begin{proposition}\label{prop:2.3}
We have for any $w\in C^2(\overline{\Omega})$, the function $u$ defined by \eqref{eq:2.abc} satisfy
 \begin{equation}\label{eq:8.70}
  S_k[u](y)=\tilde{K}+O(1)\varepsilon^{\frac{9}{2}},
 \end{equation}
where $O(1)$  depends on $\|w\|_{C^2}$ and $\tilde{K}$ is defined in \eqref{eq:8.14}.
\end{proposition}

\begin {proof}
By \eqref{2.12} and \eqref{eq:8.22}, we have
$$
\sigma_{k-1}(\lambda_1,\ldots,\lambda_{k-1})=\sigma_{k-1}(\tau)+O(1)\varepsilon^4.
$$
Noticing that $r_{ij}={P_{ij}(\varepsilon^2  x)+\varepsilon^{\frac{9}{2}}w_{ij}(x)}$ for $i\geq k$ or $j\geq k$, we obtain
$S_k(\mathbf{r})=\text{\ding {172}}+\text{\ding {173}}+\text{\ding {174}}$ with
 \begin{equation*}\left\{
\begin{array}{ll}
\text{\ding {172}}&=[\sigma_{k-1}(\tau)+\varepsilon^4O(1)] \sum_{l=k}^nr_{ll}(\varepsilon^2x)-
\sum_{l=k}^n\sum_{j=1}^{k-1}\sigma_{k-1,j}(\lambda_1,\ldots,\lambda_{k-1})r_{jl}^2(\varepsilon^2x)\\
&=\sigma_{k-1}(\tau)\sum_{l=k}^nr_{ll}(\varepsilon^2x)+\varepsilon^{8}O(1)=\sigma_{k-1}(\tau)\sum_{l=k}^nP_{ll}(\varepsilon^2x)+\varepsilon^{\frac{9}{2}}O(1)\\
\text{\ding {173}}&=\sigma_{k-2}(\lambda_1,\ldots,\lambda_{k-1})\sum_{ k\leq l_1,l_2\leq n}
  (r_{l_1l_1}(\varepsilon^2x)r_{l_2l_2}(\varepsilon^2x)-r_{l_1l_2}^2(\varepsilon^2x))\\
 & \ \ \ \ \ +{O(r_{il}^3(\varepsilon^2x))}=\varepsilon^8O(1)\\
\text{\ding {174}}&={O(|r_{il}^3(\varepsilon^2x)|)}=\varepsilon^{12}O(1).
\end{array}\right.
\end{equation*}
Therefore,
$$
S_{k}(\mathbf{r})=\sigma_{k-1}(\tau) \sum_{j=k}^nP_{jj}(\varepsilon^2x)+O(1)\varepsilon^{\frac{9}{2}},
$$
from which, using  \eqref{eq:8.13} and recalling $y=\varepsilon^2x$, we obtain \eqref{eq:8.70}.
\end{proof}


\section{Linearized degenerate elliptic operators} \label{section3}

By the construction of Section \ref{section2} and Proposition \ref{prop:2.3}, we have
\begin{equation}\label{eq:8.18}
G(w)=\frac{1}{\varepsilon^{\frac{9}{2}}}\left\{S_k(u)- \tilde{K}\right\}=\frac{1}{\varepsilon^{\frac{9}{2}}}\left\{O(1)\varepsilon^{\frac{9}{2}}\right\}=O(1),
\end{equation}
So that for any $w\in C^2(\bar\Omega)$,
\begin{equation}\label{eq:2.11a}
\theta(w)=\sup_{x\in \Omega}|G(w)(x)|<+\infty
\end{equation}
uniformly with respect to $\varepsilon$, and then \eqref{eq:8.18} is well-defined for $0<\varepsilon\le \varepsilon_0<<1$.

The linearized operator of $G$ at $w$ is
\begin{equation*}
L_{G}(w)=\sum_{i,j=1}^{n}S_{k}^{ij}(w)\partial_{i}\partial_{j},
\end{equation*}
where
\begin{equation}\label{eq:4.4}
S_{k}^{ij}(w)=\frac{\partial S_{k}}{\partial r_{ij}}(w)=\frac{\partial S_{k}(\mathbf{r})}{\partial r_{ij}}(w).	
\end{equation}
Since  the matrix $(S^{ij}_{k})(w)$ and $\mathbf{r}$ is simultaneously diagonalizable, see \cite{W},  that is,  for any
smooth function $w$, we can find  out an orthogonal matrix $T(x,\varepsilon)$ satisfying
\begin{equation}\label{eq:4.4++}
\left\{
\begin{array}{l}
T(x,\varepsilon)(S_{k}^{ij})\ \ \Prefix ^{t}{T(x,\varepsilon)}=
\textup{diag}\left[\frac{\partial \sigma_{k}(\lambda(x, \varepsilon))}{\partial
\lambda_{1}}, \frac{\partial\sigma_{k}(\lambda(x, \varepsilon))}{\partial
\lambda_{2}},\ldots,\frac{\partial\sigma_{k}(\lambda(x, \varepsilon))} {\partial
\lambda_{n}}\right]\\
T(x,\varepsilon)\mathbf{r}\ \Prefix ^{t}{T(x,\varepsilon)}=\textup{diag}\left[\lambda_{1}(x,\varepsilon),
\lambda_{2}(x,\varepsilon),\ldots,\lambda_{n}(x,\varepsilon)\right],
\end{array}\right.
\end{equation}
where $T_i(x, \varepsilon)$ is the corresponding unit eigenvectors of $\lambda_i,  i=1, 2, \ldots, n$.

The linearized operator $L_G(w)$ is not guaranteed to be degenerately elliptic, because
$(\lambda_1(x,\varepsilon),\lambda_2(x,\varepsilon),\ldots,\lambda_n(x,\varepsilon))$, as a result of the perturbation  by $\varepsilon^{\frac{9}{2}}w_{ij}(x)$, may be not in $\bar\Gamma_k.$ But we have

\begin{proposition}\label{lm:ellip}
Assume that $\|w\|_{C^{2}(\Omega)}\leq 1$ , then the second order differential operators
\begin{equation*}
L_{G}(w)+\theta\Delta
\end{equation*}
is a degenerate elliptic operator if $\varepsilon>0$ is sufficiently
small, where $\theta$ is defined in \eqref{eq:2.11a}.
\end{proposition}

\begin{proof}
By the definition of degenerate elliptic operator, we have to prove
$$A=\theta|\xi|^{2}+\sum_{i,j=1}^{n}S^{ij}_{k}\xi_{i}\xi_{j}\geq 0,\ \ \mbox{for any}\ \ \xi\in \mathbb{R}^{n}
$$
which is equivalent to prove,
$$
A=\theta|\tilde\xi|^{2}+((S^{ij}_{k})\ \Prefix ^{t}T\tilde\xi,
 \Prefix ^{t} T\tilde\xi)=\theta|\tilde\xi|^{2}+
(T(S^{ij}_{k})\ \Prefix ^{t}T\tilde\xi, \tilde\xi)\geq 0,\ \ \mbox{for any}\ \ \tilde\xi\in \mathbb{R}^{n},
$$
where $T$ is the orthogonal matrix in \eqref{eq:4.4++}, then
\begin{equation}\label{eq:4.12}
A=\theta|\tilde\xi|^{2}+\sum_{i=1}^{k-1}\sigma_{k-1,i}(\lambda(x, \varepsilon))\tilde\xi_{i}^{2}
+\sum_{i=k}^{n}\sigma_{k-1,i}(\lambda(x, \varepsilon))\tilde\xi_{i}^{2}
\end{equation}
 Since for small $\varepsilon$, we have  $\sigma_{k-1,i}(\lambda(x, \varepsilon))=\prod_{j=1}^{k-1}\tau_j+O(\varepsilon)>0, \,\, k\leq i\leq n$, we only need to prove
 \begin{equation}\label{eq:4.13}
  \theta+\sigma_{k-1,i}(\lambda(x, \varepsilon))\geq 0, \indent 1\leq i\leq k-1.
 \end{equation}
If $\sigma_k(\lambda(x, \varepsilon))\geq 0,$ together with the fact
$\sigma_{j}(\lambda(x, \varepsilon))=\sigma_j(\tau_1,\tau_2,\ldots,\tau_{k-1})+O(\varepsilon)>0$ for $1\leq j\leq k-1$, then $\lambda(x, \varepsilon)\in \bar{\Gamma}_k$
which  yields
$$
\sigma_{k-1,i}(\lambda(x, \varepsilon))\geq 0, \indent 1\leq i\leq n.
$$
It is left to consider the case $\sigma_k(\lambda(x, \varepsilon))<0,$ in which case,  since $\tilde{K}\geq 0,$
\begin{equation*}
\begin{split}
\theta&=\theta(w)=\max_{x\in \Omega}|G(w)|\\
&=\max_{x\in \Omega}\frac{1}{\varepsilon}|\sigma_k(\lambda(x, \varepsilon))-\tilde{K}|\geq -\frac{1}{\varepsilon}\sigma_k(\lambda(x, \varepsilon)).
\end{split}
\end{equation*}
 Now we prove \eqref{eq:4.13} for $i=1$ with $\sigma_{k-1,1}(\lambda)< 0$, the other cases
can be proved similarly. By the definition of $G(w)$ and $\sigma_k(\lambda)=\lambda_1\sigma_{k-1,1}(\lambda)+\sigma_{k,1}(\lambda),$
\begin{equation}\label{eq:4.15}
\begin{split}
\theta&+2\sigma_{k-1,1}(\lambda)=\theta+2\frac{\sigma_k(\lambda)-\sigma_{k,1}(\lambda)}{\lambda_1}\\
\geq& -\frac{1}{\varepsilon}\sigma_k(\lambda)
+2\frac{\sigma_k(\lambda)-\sigma_{k,1}(\lambda)}{\lambda_1}=(-\frac{1}{\varepsilon}+\frac{2}{\lambda_1})\sigma_k(\lambda)-\frac{2}{\lambda_1}
\sigma_{k,1}(\lambda).
\end{split}
\end{equation}

Under the assumption $\sigma_{k-1,1}(\lambda)<0$ and $\sigma_k(\lambda)<0$, we will distinguish two cases.

\noindent{\bf{Case 1.}} If $\sigma_{k,1}(\lambda)\leq0$, we have by \eqref{eq:4.15}
$$
\theta+\sigma_{k-1,1}> \theta+2\sigma_{k-1,1}\geq (-\frac{1}{\varepsilon}+\frac{2}{\lambda_1})\sigma_k(\lambda)-\frac{2}{\lambda_1}
\sigma_{k,1}(\lambda)\geq (-\frac{1}{\varepsilon}+\frac{2}{\lambda_1})\sigma_k(\lambda)>0
$$
provided $\varepsilon$ is small enough.

\noindent
{\bf{Case 2.}} Next it is left to consider the case in which
 $$
\sigma_{k}(\lambda)<0,\,\,\,\sigma_{k,1}(\lambda)>0 ,\,\,\,\sigma_{k-1,1}(\lambda)<0
$$
hold simultaneously. Using Newton's inequalities for $(n-1)$-tuple
vectors
$$
\sigma_{k,1}(\lambda)\sigma_{k-2,1}(\lambda)\leq \frac{(k-1)(n-k)}{k(n-k+1)}[\sigma_{k-1,1}(\lambda)]^2, \indent \lambda\in \mathbb{R}^n
$$
and the fact
$$
\sigma_{k-2,1}(\lambda)=\prod_{i=2}^{k-1}\tau_i+O(\varepsilon)>0,\indent \sigma_{k-1,1}(\lambda)=O(\varepsilon),
$$
we obtain
$$
0< \sigma_{k,1}\leq \frac{(k-1)(n-k)}{k(n-k+1)}\frac{[\sigma_{k-1,1}(\lambda)]^2}{\sigma_{k-2,1}(\lambda)}\leq |O(\varepsilon)\sigma_{k-1,1}(\lambda)|.
$$
Back to \eqref{eq:4.15}, using $\sigma_{k}(\lambda)<0$, then for $\varepsilon>0$ small,   we have
$$
\theta+2\sigma_{k-1,1}\geq \left(-\frac{1}{\varepsilon}+\frac{2}{\lambda_1}\right)\sigma_k(\lambda)-\frac{2}{\lambda_1}
\sigma_{k,1}(\lambda)\geq -\frac{2}{\lambda_1}
\sigma_{k,1}(\lambda)= -|O(\varepsilon)\sigma_{k-1,1}(\lambda)|,
$$
which yields
$$
\theta+\sigma_{k-1,1}\geq \theta+2\sigma_{k-1,1}+|O(\varepsilon)\sigma_{k-1,1}(\lambda)|\geq0
$$
provided $\varepsilon>0$ small enough.  Proof is done.
\end{proof}

Equality \eqref{eq:4.12} shows that the operator $L_{G}(w)+\theta\Delta$ may be degenerate elliptic in the directions of  $(\tilde{\xi}_1,\cdots,\tilde{\xi}_{k-1})$ and is uniformly elliptic in the directions of  $(\tilde{\xi}_k,\cdots,\tilde{\xi}_{n})$ after the orthogonal transform $T(x,\varepsilon)$ which is a perturbation of unit matrix, so we can impose the Dirichlet boundary condition on $\partial Q_{\delta_0}$ (the $x''$ direction), but we can't do that  on $\partial Q_{\pi}$ (the $x'$ direction). Instead of treating a Dirichlet boundary value problem, we shall  prove the  existence, uniqueness and \`a priori estimates of  the solution, which is periodic with respect to $x',$  to the  degenerately  linearized elliptic equation
\begin{equation}\label{eq:4.5}\left\{
\begin{array}{l}
L_{G}(w)\rho+\theta\Delta\rho=g,  \mbox{in} \ \ \Omega;\\
\rho(x', x'')\,\,  \mbox{is periodic for}\,\,  x'\in Q_{\pi}\,\,\mbox{and}\,\,   \rho(x', x'')=0 \ \ \mbox{for} \,\, x''\in\partial Q_{\delta_0},
\end{array}\right.
\end{equation}
in some suitable Hilbert space  defined below, this idea is inspired by Hong and Zuily \cite{HZ} where they consider the case  $k=n$.  We introduce the space $\mathbf{H}^{s}(\Omega)$ ($s$ is an  integer), which is the completion of the space of trigonometrical
polynomials
$$
\rho(x)=\sum_{\ell=(l_1,l_2,\ldots,l_{k-1})\in \mathds{Z}^{k-1}}\alpha_{\ell}(x'')e^{\sqrt{-1}\sum_{j=1}^{k-1}l_jx_j}
$$
with the (complex-valued) coefficients $\alpha_{\ell}$  subject to the condition  $\overline{\alpha_{\ell}}(x'')=\alpha_{-\ell}(x'')\in C^\infty(Q_{\delta_0})$, with respect  to the norm
$$
\|\rho\|_{s}^{2}=\sum_{t+j\leq s}\sum_{\ell\in \mathds{Z}^{k-1}}(1+\sum_{i=1}^{k-1}l_i^2)^t\|\alpha_{\ell}\|_{H^j(Q_{\delta_0})}^{2},
$$
where $H^j(Q_{\delta_0})$ is the usual Sobolev space.   We define $\mathbf{H}^{s}_0(\Omega)$ in the same way by taking $\alpha_{\ell} \in C^\infty_0(Q_{\delta_0})$. We will prove, in the next section, the following Theorem.

\begin{theorem}\label{th:linearnore}
Let  $w$ be smooth and $\|w\|_{C^{[\frac{n}{2}]+3+l_0}(\Omega)}\leq 1$ with nonnegative integer
$l_0.$  Then for any $ s_{0}\in \mathbb{N}\cup\{0\}$,  one can find a constant
$\varepsilon(s_{0})$ such that the equation \eqref{eq:4.5} possesses a  solution $\rho\in \mathbf{H}^s_0(\Omega)$ provided that $g\in {H}^{s}(\Omega)$, $0\leq s\leq s_{0}$ and $0<\varepsilon\leq \varepsilon(s_{0})$.  If $s\geq 1,$ the solution is unique .   Moreover,
\begin{equation}\label{eq:5.38}
\left\{\begin{array}{l}
\|\rho\|_{s}\leq C_{s}(\|g\|_{s}+\sum_{i,j=1}^n\|W_{ij}\|_{s+2}\|\rho\|_{L^{\infty}}),\quad \mbox{if}\,\, s>\left[\frac{n}{2}\right]+1+l_0;\\
\|\rho\|_{s}\leq C_{s}\|g\|_{s},\,\quad  \mbox{if}\,\, s\le \left[\frac{n}{2}\right]+1+l_0
\end{array}\right.
\end{equation}
holds for some constants $C_{s}$ independent of $w$ and $\varepsilon$. Here
$$
W_{ij}(x)=P_{y_iy_j}(\varepsilon^2x)+
 \varepsilon^{\frac{9}{2}} w_{y_iy_j}(x).
$$
\end{theorem}

\begin{remark}
Since the equation \eqref{eq:4.5} is degenerate elliptic, we can only get the a priori estimate \eqref{eq:5.38}  with a loss of order $2$. This loss of regularity of linearized equation ask us to use the
Nash-Moser-H\"ormander iteration to deal with the solution of  \eqref{eq:4.5}. By definition of  $r_{ij}$ by \eqref{2.abd}, if $P(y)=0$, then $\sum_{i,j=1}^n\|(W_{ij})\|_{s+2}$ is reduced to $\|(w)\|_{s+4}$ which is introduced in Theorem 1.3 of \cite{HZ}. When $l_0=0$, Theorem \ref{th:linearnore} $(2\leq k\leq n)$ is a generalization of Theorem 1.3 $( k=n)$ in \cite{HZ}.
The assumption $\|w\|_{C^{[\frac{n}{2}]+4}(\Omega)}\leq 1 (l_0=1)$ is necessary  in the estimates of the quadratic error in Lemma \ref{lm:du} for $f=f(y,u,Du)$ in Lemma \ref{lm:du}, although we will not give its estimates of the quadratic error ; while the assumption $\|w\|_{C^{[\frac{n}{2}]+3}(\Omega)}\leq 1 (l_0=0)$ is enough in case $f=f(y,u)$. The uniqueness for $s\geq 1$ follows from \eqref{eq:4.37} by taking $\nu=0.$
\end{remark}


\section{\`A priori estimates of solutions for linearized equations} \label{section4}

First of all, using the change of unknown function $\bar{\rho}=\rho
e^{\mu\sum_{j=k}^nx_{j}^{2}}$, we reduce \eqref{eq:4.5} to
\begin{equation}\label{eq:4.18}
\left\{
\begin{array}{l}
L(w)\bar\rho=\sum_{i,j=1}^{n}(S^{ij}_{k}(w)+
\delta_{i}^{j}\theta)\partial_{i}\partial_{j}\bar\rho+\sum_{i=1}^{n}b_{i}
\partial_{i}\bar\rho+c\bar\rho=e^{\mu \sum_{j=k}^nx_{j}^{2}}g, \ \ \  \mbox{in} \ \ \Omega\\
\bar{\rho}=0 \ \ \mbox{on} \ \ \partial Q_{\delta_0}\  \mbox{and periodic on} \ \ Q_{\pi}.
\end{array}\right.
\end{equation}
The coefficients $b_{i}$ and $c$ are expressed as follows.
$$
b_i=\left\{
\begin{array}{l}
-4\sum_{j=k}^n (\mu x_j) S^{ij}_k(w) ,\indent 1\leq i\leq k-1\\
-4(\mu x_i)\theta-4 \sum_{j=k}^n (\mu x_j) S^{ij}_k (w),\indent k\leq i\leq n.
\end{array}\right.
$$
$$
  c=-2\mu\sum_{i=k}^n(S^{ii}_k(w)+\theta)
  +4\sum_{i=k}^n\sum_{j=k}^n (\mu x_i)(\mu x_j) S^{ij}_k(w)+4\theta\sum_{i=k}^n (\mu x_i)^2.
$$
We would like to  prove Theorem \ref{th:linearnore} for \eqref{eq:4.18}  rather than  \eqref{eq:4.5}, and write $\rho$ instead of $\bar\rho,$ but also not do it directly, we will  consider  the solution $\rho_\nu$ to
 the regularized version of \eqref{eq:4.18}, i.e., the  following uniformly elliptic  equation, for $ 0<\nu<1$,
\begin{equation}\label{eq:4.38}\left\{
\begin{array}{l}
L_{\nu}\rho\equiv\sum_{i,j=1}^{n}(S^{ij}_{k}+
\delta_{i}^{j}\theta)\partial_{i}\partial_{j}\rho+\nu\triangle \rho+\sum_{i=1}^{n}b_{i}
\partial_{i}\rho+c\rho=g,  \mbox{in} \ \ \Omega,\\
\rho=0 \ \ \mbox{on} \ \ \partial Q_{\delta_0}\  \mbox{and periodic on} \ \ Q_{\pi}.
\end{array}\right.
\end{equation}

We first need the following Lemmas which is standard for the degenerate elliptic operators. So we only point out some important
steps for the proof.

\begin{lemma}\label{th:s2}
Suppose that $w$ is smooth and  $\|w\|_{C^4(\Omega)}\leq 1.$
Then there exists two positive constants  $\mu_0$ large and
$\varepsilon_{0}$ small such that, for $0<\varepsilon\leq \delta_{0}$, $ \mu_0\delta_0\leq 1,$
  $g\in {H}^0(\Omega)$  and  $\nu>0,$ problem \eqref{eq:4.38} admits an unique solution $\rho_\nu\in  \mathbf{H}^1_0(\Omega)$ , which satisfies
\begin{equation}\label{eq:4.17}
\|\rho_\nu\|_{0}\leq C_0\|g\|_{0},
\end{equation}
where $C_0$ is uniform for  $\nu\in ]0, 1[, \varepsilon\in ]0, \delta_0]$ and independent of $w.$
\end{lemma}

\begin{proof}
We prove the existence and uniqueness of the solution $\rho_\nu$ to \eqref{eq:4.38}  by applying Lax-Milgram Theorem to  the bilinear form
$$
<-L_\nu \rho,\varrho>
$$
where $<\cdot,\cdot>$ is the dual pair on $H^{-1}\times H^1_0.$   The condition $\|w\|_{C^{4}}\leq 1$  yields
$$
|<-L_{\nu}\rho,\varrho>|\leq C\|\rho\|_1\|\varrho\|_1, \indent \forall\rho,\varrho\in H^1_0,
$$
where $C$ is uniform on $0<\varepsilon<1,0<\nu<1$.  For the coercivity, the proof of which is almost the same as that of Lemma 1.4, \cite{HZ}, there exist $\varepsilon_0>0$ small and large $\mu>0$
such that
\begin{equation}\label{eq:4.37}
-<L_{\nu}\rho,\rho> \,\,\geq \nu \|D\rho\|^2_0+ 2\sigma_{k-1}(\tau)\|\rho\|^{2}_0,
\end{equation}
then by using Lax-Milgram Theorem, for $g\in {H}^0(\Omega)$  and  $\nu>0,$ problem \eqref{eq:4.38} admits an unique solution $\rho_\nu\in  \mathbf{H}^1_0(\Omega)$. Since $|-<L_{\nu}\rho_\nu,\rho_\nu>|=|<g,\rho_\nu>|\leq \|g\|_0\|\rho_\nu\|_0$, then \eqref{eq:4.17} follows from \eqref{eq:4.37}.
\end{proof}

Similarly to the proof of Theorem 1.3, \cite{HZ},  we have the higher order \`a priori estimates

\begin{lemma}\label{lm:4.2}
Suppose that $w$ is smooth and  $\|w\|_{C^{[\frac{n}{2}]+3}(\Omega)}\leq 1.$
 For $ s>0$,  then there exists
$\varepsilon_{0}(s)>0$ small such that, for $0<\varepsilon\leq \varepsilon_{0}(s),$ and $g\in {H}^s(\Omega)$,
the problem \eqref{eq:4.38} admits an unique solution $\rho_\nu\in  \mathbf{H}^{s+1}_0(\Omega)$ , which satisfies
\begin{equation}\label{eq:4.17+}
\left\{\begin{array}{l}
\|\rho_\nu\|_{s}\leq C_{s}(\|g\|_{s}+\sum_{i,j=1}^n\|W_{ij}\|_{s+2}\|\rho_\nu\|_{L^{\infty}}),\,\,\,\mbox{if}\,\, s>\left[\frac{n}{2}\right]+1;\\
\|\rho_\nu\|_{s}\leq C_{s} \|g\|_{s}, \,\,\,\,\, \mbox{if}\,\, s\le \left[\frac{n}{2}\right]+1
\end{array}\right.
\end{equation}
where $C_s$ is uniform for  $\nu\in ]0, 1], \varepsilon\in ]0,\varepsilon_0(s)]$ and independent of $w$.
\end{lemma}

\begin{proof}  For $s=0$, Since $-(L_{\nu}\rho,\rho)=-(g,\rho),$ then \eqref{eq:4.37} and Cauchy inequality yields
$$
\nu\|D\rho\|_0^2+\|\rho\|_0^2\leq C_{0}(\tau)\|g\|^{2}_{0},
$$
where $C_0(\tau)$ is independent of $\nu,w$ and $\varepsilon.$  On the other hand, for $\alpha\in\mathbb{N}^n, |\alpha|\le s$
$$
(L_{\nu}(\partial^\alpha \rho),(\partial^\alpha \rho))=(\partial^\alpha g, (\partial^\alpha \rho))+([L_{\nu},\partial^{\alpha}]\rho, (\partial^\alpha \rho)),
$$
where the commutators is
$$
[L_{\nu},\partial^{\alpha}]=-\sum_{\beta\leq \alpha,|\beta|\geq1}\sum_{i,j=1}^nC_{\alpha\beta}\Big(\partial^{\beta}(S^{ij}_{k})\partial_{i}\partial_{j}
+\sum_{i=1}^n\partial^{\beta}b_{i}\partial_{i}
+\partial^{\beta}c\Big)\partial^{\alpha-\beta},
$$
here the coefficients depends on $D^2 w$, by using the interpolation inequalities, we can get:

\noindent
{\bf 1)} If $s\le \left[\frac{n}{2}\right]+1$, then the condition $\|w\|_{C^{[\frac{n}{2}]+3}(\Omega)}\leq 1$ imply
$$
|([L_{\nu},\partial^{\alpha}]\rho, (\partial^\alpha \rho))|\le \varepsilon C \|\rho\|^2_s.
$$

\noindent
{\bf 2)} If $s> [\frac{n}{2}]+1$, a little involved computation also give
$$
|([L_{\nu},\partial^{\alpha}]\rho, (\partial^\alpha \rho))|\leq
C_s\Big(\varepsilon\|\rho\|_{s}^{2}+\|\rho\|_s\big(\|g\|_s+\|\rho\|_{L^{\infty}}\sum_{i,j=1}^n\|W_{ij}\|_{s+2}\big)\Big),
$$
where $C_s$ depends only on $s$, so we finish the proof.
\end{proof}

With Lemmas \ref{th:s2} and \ref{lm:4.2}, we can now prove Theorem  \ref{th:linearnore}.

\begin{proof}  {\bf The proof of Theorem  \ref{th:linearnore}.}  To simplify the computation, we consider only the case $l_0=0$, and prove \eqref{eq:5.38}
 under the assumption $\|w\|_{C^{[\frac{n}{2}]+3}(\Omega)}\leq 1.$
 Now for  $0\leq s\leq [\frac{n}{2}]+1$, we can apply  Banach-Saks Theorem to find a subsequence
$\rho_{\nu_j}$ with $\nu_j=j^{-2}$ and an element $\rho_0\in \mathbf{H}^s_0$ such that
$$
\|\frac{\rho_{\nu_1}+\cdots+\rho_{\nu_m}}{m}-\rho_0\|_s\rightarrow 0\,\, (m\rightarrow \infty).
$$
Since  $\frac{\rho_{\nu_1}+\cdots+\rho_{\nu_m}}{m} $ is periodic in $x'$ for each $m,$ so is $\rho_0$. Back to \eqref{eq:4.38}, we have
$$
\left[\sum_{i,j=1}^{n}(S^{ij}_{k}+
\delta_{i}^{j}\theta)\partial_{i}\partial_{j}+\sum_{i}^{n}b_{i}
\partial_{i}+c\right]\frac{\rho_{\nu_1}+\cdots+\rho_{\nu_m}}{m}+\frac{1}{m}\sum_{j=1}^m\nu_j\triangle \rho_{\nu_j}=g.
$$
For any test function $\phi\in C^\infty_0$, Lemma \ref{th:s2} yields
$$
|(\frac{1}{m}\sum_{j=1}^m\nu_j\triangle \rho_{\nu_j}, \phi)_{L^2}|\le \|\triangle \phi\|_0\, \|\frac{1}{m}\sum_{j=1}^m\nu_j\rho_{\nu_j}\|_0\leq C_0
 \|\triangle \phi\|_0\, \|g\|_0\,\frac{1}{m}\sum_{j=1}^m\nu_j\rightarrow 0,
$$
taking $m\rightarrow \infty,$  we have that $\rho_0$ is a solution of \eqref{eq:4.5} in the sense of  distribution :
$$
\sum_{i,j=1}^{n}(S^{ij}_{k}+
\delta_{i}^{j}\theta)\partial_{i}\partial_{j}\rho_0+\sum_{i}^{n}b_{i}\partial_{i}\rho_0+c\rho_0=g.
$$
Moreover, by Lemma \ref{lm:4.2}
$$
\|\rho_0\|_{s}=\lim_{m\rightarrow \infty}\|\frac{\rho_{\nu_1}+\cdots+\rho_{\nu_m}}{m}\|_s\leq C_{s}\|g\|_{s}, \ \ \  \mbox{for} \ \ 0\leq s\leq [\frac{n}{2}]+1.
$$
Now we prove \eqref{eq:5.38} for $s>[\frac{n}{2}]+1$. Since Lemma \ref{lm:4.2} yields
$$
\|\rho_\nu\|_{[\frac{n}{2}]+1}\leq C_{[\frac{n}{2}]+1}\|g\|_{[\frac{n}{2}]+1},
$$
by Sobolev imbedding theorem , $\|\rho_\nu\|_{L^\infty}< \infty$, moreover if  $\sum_{i,j=1}^n\|W_{ij}\|_{s+2}<\infty$, then \eqref{eq:4.17+} shows that
$$
\|\rho_\nu\|_{s}\leq C<\infty \ \ \mbox{for}\ \ \ s>[\frac{n}{2}]+1,
$$
then, by weak compactness theorem,  there is a subsequence $\rho_{v_j}$ of $\rho_v$ and $\rho_0$ such that
$\rho_{\nu_j}\rightarrow \rho_0$
in the weak topology of $W^{s,2}(\Omega)$ and
$$
\|\rho_{\nu_j}-\rho_0\|_{L^\infty(\Omega)}\rightarrow 0,\|\rho_{\nu_j}-\rho_0\|_{W^{1,2}(\Omega)}\rightarrow 0,
$$
therefore $\rho_0$ is periodic in $x'$ , $\rho_0\in H^s_0(\Omega),$  and by  Lemma \ref{lm:4.2} again
$$
\|\rho_{0}\|_{s}\leq \liminf_{j\rightarrow \infty}\|\rho_{\nu_j}\|_{s}\leq C_{s}(\|g\|_{s}+\sum_{i,j=1}^n\|W_{ij}\|_{s+2}\|\rho_0\|_{L^{\infty}}),
$$
so we complete the proof of \eqref{eq:5.38}.
\end{proof}


\section{Nash-Moser-H\"ormander iteration}\label{section5}

We prove firstly the existence of $k$-convex local solution of \eqref{eq:8.8++} as a perturbation of $\psi$ constructed in Section \ref{section2} by employing the Nash-Moser-H\"{o}mander  iteration, which is based on  the \`a priori estimates established in last section.  Since the linearized operators is degenerate elliptic,  a loss of regularity of order $2$ has occurred for the solution of linearized equation, so we need to {\bf mollify} the solution by taking  a family of smoothing operators $S(t), t\geq 1$ such that
$$
S(t): \, \cup_{s\geq 0}\mathbf{H}^s_0(\Omega)\,\longrightarrow \,\cap_{s\geq 0}\mathbf{H}^s(\Omega)
$$
with the following properties:
\begin{equation}\label{eq:6.5+}
\|S(t)u\|_{s_1}\leq C_{s_1s_2}\|u\|_{s_2}, \indent \text{if} \indent s_1\leq s_2,
\end{equation}
\begin{equation}\label{eq:6.6+}
 \|S(t)u\|_{s_1}\leq C_{s_1s_2}t^{s_1-s_2}\|u\|_{s_2}, \indent \text{if}\indent s_1\geq s_2,
\end{equation}
\begin{equation}\label{eq:6.7+}
 \|S(t)u-u\|_{s_1}\leq C_{s_1s_2}t^{s_1-s_2}\|u\|_{s_2}, \indent \text{if}\indent s_1\leq s_2,
\end{equation}
where $C_{s_1s_2}$  is independent of $t$ and depends only on $s_1,s_2$ , see \cite{AG} for more detailed properties of smoothing operators.

Now we define $S_m=S(\mu_m)$ with $\mu_m=\sigma^{\gamma^m}$ where $\sigma>1, \gamma>1$ to be determined later, we use the following iteration procedure:
\begin{equation}\label{eq:6.9}
\left\{
\begin{array}{l}
 w_0=0,w_{m+1}=w_m+S_m\rho_m, \\
 L(w_m)\rho_m=L_G(w_m)\rho_m+\theta_m\triangle \rho_m=g_m, \indent x\in \Omega\\
 \rho_m\in \mathbf{H}^s_0(\Omega),g_m\in H^s(\Omega).
\end{array}\right.
\end{equation}
where $g_m$ and $\theta_m$ are defined in \eqref{eq:6.10},  $G(w)$ is defined in \eqref{eq:8.18}.

Remember we are studying the equation \eqref{eq:4.4} in the variables $x$ rather than $y=\varepsilon^2 x$,
and
\begin{equation*}
u_m(y)=\frac{1}{2}\sum_{j=1}^{k-1}\tau_jy_j^2+P(y)+\varepsilon^{\frac{17}{2}} w_m(\varepsilon^{-2}y).
\end{equation*}
Denoting
\begin{align*}
&\mathcal{M}(s)=\|g_0\|_s+\varepsilon^4\sum_{i,j=1}^n\|P_{ij}(\varepsilon^2\ \cdot)\|_{s-1},\\
&\mathcal{N}(s)=\mathcal{M}(s)+\sum_{i,j=1}^n\|P_{ij}(\varepsilon^2 \ \cdot)\|_{s+2}\Big[1+\mathcal{M}\Big(\left[\frac{n}{2}\right]
+1\Big)\Big]
\end{align*}
which are small if $\varepsilon>0$ small. We prove the \`a priori estimates by induction.

\begin{lemma}\label{lemma4.3}
Suppose that  $\|w_l\|_{C^{[\frac{n}{2}]+3}}\leq 1$   for $0\le l\le m$. Then
\begin{equation}\label{eq:6.11}
  \|g_m\|_s\leq C_s(\mathcal{M}(s)+\|w_m\|_{s+2})
\end{equation}
and
\begin{equation}\label{eq:6.12}
  \|w_{m+1}\|_{s+4}\leq C_s^{m+1}\mu_{m+1}^\beta \mathcal{N}(s), \indent \text{for}\indent \beta=\frac{4}{\gamma-1},
  \end{equation}
where $C_s$ is independent of $m$ and $\gamma.$
\end{lemma}

\begin{proof}  We prove first \eqref{eq:6.11}. Remembering
$$
\mathbf{r}_m=\left(\sum_{l=1}^{k-1}\delta_{i}^{j}\delta_j^l\tau_l+ P_{ij}(\varepsilon^2x)+
 \varepsilon^{\frac{9}{2}} (w_m)_{ij}(x)\right)
$$
using Taylor expansion and $w_0=0$, we have
$$
-g_m=G(w_m)=G(w_0)+\int^1_0\frac{\partial}{\partial t}[G(t w_m)]d t=-g_0+\int^1_0\sum_{i,j=1}^nS_k^{ij}(\mathbf{r}_m(t))\frac{\partial^2w_m}{\partial x_i\partial x_j} dt.
$$
 Hence
\begin{align*}
\|g_m\|_s&\le \|g_0\|_s+\sum_{i,j=1}^n\sum_{|\alpha|+\beta|\le s}C^\beta_\alpha
\|\partial^{\alpha}S_k^{ij}(\mathbf{r}_m)\partial^{\beta}(\frac{\partial^2w_m}{\partial x_i\partial x_j})\|_0.
\end{align*}
By using $\|w_m\|_{C^{[\frac{n}{2}]+3}}\leq 1$, we have
$$
\sum_{|\alpha|+|\beta|\leq s, |\alpha|\leq [\frac{n}{2}]+1}\|\partial^{\alpha}S_k^{ij}(\mathbf{r}_m)\partial^{\beta}(\frac{\partial^2w_m}{\partial x_i\partial x_j})\|_0\leq C\sum_{|\beta|\leq s}\|\partial^{\beta}(\frac{\partial^2w_m}{\partial x_i\partial x_j})\|_0\leq C\|w_m\|_{s+2}.
$$
On the other hand, by interpolation and $\|w_m\|_{C^{[\frac{n}{2}]+3}}\leq 1$,
\begin{align*}
&\sum_{|\alpha|+|\beta|\leq s, |\alpha|> [\frac{n}{2}]+1}\|\partial^{\alpha}S_k^{ij}(\mathbf{r}_m)\partial^{\beta}(\frac{\partial^2w_m}{\partial x_i\partial x_j})\|_0\\
=&\sum_{|\alpha|+|\beta|\leq s, |\alpha|> [\frac{n}{2}]+1}\|\partial^{\alpha-[\frac{n}{2}]-1}\partial^{[\frac{n}{2}]+1}S_k^{ij}(\mathbf{r}_m)\partial^{\beta}(\frac{\partial^2w_m}{\partial x_i\partial x_j})\|_0\\
\leq & C\sum_{|\alpha|+|\beta|\leq s, |\alpha|> [\frac{n}{2}]+1}\left[\|\partial^{[\frac{n}{2}]+1}S_k^{ij}(\mathbf{r}_m)\|_{L^\infty}\|\partial_i\partial_jw_m\|_{|\alpha|
+|\beta|-[\frac{n}{2}]-1}\right.\\
& \ \ \ \ \ \left.+\|\partial^{[\frac{n}{2}]+1}S_k^{ij}(\mathbf{r}_m)\|_{|\alpha|
+|\beta|-[\frac{n}{2}]-1}\|\partial_i\partial_jw_m\|_{L^\infty}\right]\\
\leq& C\sum_{|\alpha|+|\beta|\leq s, |\alpha|> [\frac{n}{2}]+1}\left[\|w_m\|_{|\alpha|
+|\beta|-[\frac{n}{2}]+1}+\|\partial^{[\frac{n}{2}]+1}S_k^{ij}(\mathbf{r}_m)\|_{|\alpha|
+|\beta|-[\frac{n}{2}]-1}\right]\\
\leq&  C\|w_m\|_{s-[\frac{n}{2}]+1}+C\sum_{i,j=1}^n\|(W_m)_{ij}\|_{s},
\end{align*}
where
$$
(W_m)_{ij}(x)=P_{y_iy_j}(\varepsilon^2x)+ \varepsilon^{\frac{9}{2}} (w_m)_{y_iy_j}(x).
$$
Therefore
\begin{align*}
&\|S_k^{ij}(\mathbf{r}_m)\frac{\partial^2w_m}{\partial x_i\partial x_j}\|_s\\
&\leq C\|w_m\|_{s+2}+C\|w_m\|_{s-[\frac{n}{2}]+1}+C\sum_{i,j=1}^n\|(W_m)_{ij}\|_{s}\\
&\leq C\|w_m\|_{s+2}+C\sum_{i,j=1}^n\|(W_m)_{ij}\|_{s}.
\end{align*}
Thus
\begin{align*}
&\|g_m\|_s=\|G(w_m)\|_s\\
&\leq \|G(w_0)\|_s+\int^1_0\|\frac{\partial}{\partial t}[G(t w_m)]\|_sd t \leq \|g_0\|_s+C\|w_m\|_{s+2}+C\|\partial^3 u_m\|_{s-2}.
\end{align*}
Remembering
$$
\partial_x\partial_{x_i}\partial_{{x_j}}[u_m(y)\mid_{y=\varepsilon^2 x}]=\varepsilon^4\partial_x (W_m)_{ij}(x)=\varepsilon^4\partial_x[P_{ij}(\varepsilon^2x)+\varepsilon^{\frac{9}{2}}(w_m)_{ij}(x)],
$$
we have
$$
\|g_m\|_s\leq \|g_0\|_s+C\varepsilon^4\sum_{i,j=1}^n\|P_{ij}(\varepsilon^2\ \cdot)\|_{s-1}+C\|w_m\|_{s+2}\leq C_s(\mathcal{M}(s
)+\|w_m\|_{s+2}),
$$
this completes the proof of \eqref{eq:6.11}.

We prove now \eqref{eq:6.12}. By using \eqref{eq:6.6+} and \eqref{eq:6.9}, we have
$$
\|w_{m+1}\|_{s+4}\leq \|w_{m}\|_{s+4}+\|S_m\rho_m\|_{s+4}\leq \|w_{m}\|_{s+4}+C_s\mu_m^4\|\rho_m\|_s.
$$
The \`a priori estimate \eqref{eq:5.38}  and  Sobolev imbedding theorem  yield
\begin{equation}\label{eq:6.8}
\begin{split}
\|\rho_m\|_s&\leq  C(\|g_m\|_{s}+\sum_{i,j=1}^n\|(W_{ij})\|_{s+2}\|\rho_m\|_{L^{\infty}})\\
&\leq  C(\|g_m\|_s+\sum_{i,j=1}^n\|(W_{ij})\|_{s+2}\|\rho_m\|_{[\frac{n}{2}]+1})\\
&\leq  C(\|g_m\|_s+\sum_{i,j=1}^n\|(W_{ij})\|_{s+2}\|g_m\|_{[\frac{n}{2}]+1})\\
&\leq C\Big(\mathcal{M}(s)+\|w_{m}\|_{s+2}+C\big(\sum_{i,j=1}^n\|P_{ij}(\varepsilon^2 \ \cdot)\|_{s+2}+\|w_m\|_{s+4}\big)\\
&\qquad\qquad\times \big(\mathcal{M}([\frac{n}{2}]+1)+\|w_m\|_{[\frac{n}{2}]+3}\big)\Big)\\
&\leq C\left(\mathcal{N}(s)+\|w_m\|_{s+4}\right).
\end{split}
\end{equation}
Thus we get
\begin{equation}\label{6.9+}
\|w_{m+1}\|_{s+4}\leq \|w_{m}\|_{s+4}+C_s\mu_m^4\|\rho_m\|_s\leq C\mu_m^4(\mathcal{N}(s)+\|w_m\|_{s+4}),
\end{equation}
then by using $w_0=0$ and $\mu_m=\sigma^{\gamma^m},$ the iteration of \eqref{6.9+} gives
$$
\|w_{m+1}\|_{s+4}\leq C_s^m(m+1)\sigma^{4\frac{\gamma^{m+1}-1}{\gamma-1}}\mathcal{N}(s),
$$
if we choose $C_s>1$, $C'_s=2C_s,$  and $\beta=\frac{4}{\gamma-1}$, then
$$
\|w_{m+1}\|_{s+4}\leq (C'_s)^{m+1}\mu_{m+1}^{\beta}\mathcal{N}(s).
$$
This completes the proof of \eqref{eq:6.12}.
\end{proof}

\begin{lemma}\label{lm:du}
Suppose  $\|w_l\|_{C^{[\frac{n}{2}]+4}}\leq 1$ for $l=0,1,\ldots,m$. Then for any given $s^*>s > 1$, there exists $\sigma>1, \gamma>1$ and $a>0$ such that
\begin{equation}\label{eq:6.13}
  \|g_{m+1}\|_0+ \|g_{m+1}\|_{L^\infty}\leq \mu_{m+1}^{-a}\mathcal{N}(s^*).
\end{equation}
\end{lemma}
\begin{proof}
By Taylor formula with remainder and \eqref{eq:6.9},  we consider the quadratic error of $G(w_{m+1})$
\begin{align*}
&-g_{m+1}=G(w_{m+1})=G(w_m+S_m\rho_m)\\
&=G(w_m)+L_G(w_m)S_m\rho_m+\int^1_0(1-t)\frac{\partial^2}{\partial t^2}[G(w_m+t S_m\rho_m)]d t\\
&=[G(w_m)+L(w_m)\rho_m]+L(w_m)(S_m-I)\rho_m-\theta_m\triangle S_m\rho_m\\
&\quad +\int^1_0(1- t)\frac{\partial^2}{\partial t^2}[G(w_m+ t S_m\rho_m)]d t\\
&=L(w_m)(S_m-I)\rho_m-\theta_m\triangle S_m\rho_m+\int^1_0(1- t)\frac{\partial^2}{\partial t^2}[G(w_m+ t S_m\rho_m)]d t\\
&=(A)+(B)+\int^1_0(1-t)(C)d t\,.
\end{align*}
For the last terms, we have
\begin{align*}
&(C)=\frac{\partial^2}{\partial t^2}[G(w_m+ t S_m\rho_m)]\\
=&\sum_{i,j,p,l=1}^n\frac{\partial^2S_k}{\partial r_{ij}\partial r_{pl}}(w_m+ t S_m\rho_m)
(S_m\rho_m)_{ij}(S_m\rho_m)_{pl}
\end{align*}

Noticing $\frac{\partial^2S_k}{\partial r_{ij}\partial r_{pl}}(\mathbf{r})$ is a polynomial of
order $k-2$, we have, by Sobolev imbedding theorem, \`a priori estimate \eqref{eq:5.38}, \eqref{eq:6.11} and \eqref{eq:6.6+},
$$
\| S_m\rho_m\|_{[\frac{n}{2}]+3}\leq C\mu_m^2\| \rho_m\|_{[\frac{n}{2}]+1}\leq C\mu_m^2\| g_m\|_{[\frac{n}{2}]+1}\leq C\mu_m^2(\mathcal{M}([\frac{n}{2}]+1)+\|w_m\|_{[\frac{n}{2}]+3})\leq C\mu_m^2,
$$
and
\begin{align*}
&\|\frac{\partial^2S_k}{\partial r_{ij}\partial r_{pl}}(w_m+ t S_m\rho_m)
\|_{L^\infty}
\leq \|D^2w_m+t D^2S_m\rho_m\|_{L^\infty}^{k-2}\\
& \leq  [\|D^2w_m\|_{L^\infty}+\| D^2S_m\rho_m\|_{L^\infty}]^{k-2}
\leq [1+\| S_m\rho_m\|_{[\frac{n}{2}]+3}]^{k-2}\\
 &\leq (C')^{k-2}\mu_m^{2(k-2)}.
\end{align*}
Then \eqref{eq:6.6+}   yields
\begin{align*}
\|(C)\|_0\leq &\sum_{i,j,p,l=1}^n\|\frac{\partial^2S_k}{\partial r_{ij}\partial r_{pl}}(w_m+ t S_m\rho_m)\|_{L^\infty}
\|(S_m\rho_m)_{ij}\|_{L^\infty}\|(S_m\rho_m)_{pl}\|_0\\
\leq & C^{k-2}\mu_m^{2(k-2)}\|S_m\rho_m\|_{[\frac{n}{2}]+3}\|S_m\rho_m\|_2
\leq  C^{k-2}\mu_m^{2(k-2)+[\frac{n}{2}]+5}\|\rho_m\|_0^2\\
\leq & C^{k-2}\mu_m^{2(k-2)+[\frac{n}{2}]+5}\|g_m\|_0^2.
\end{align*}
By  \eqref{eq:6.7+},  \eqref{eq:6.8} and \eqref{eq:6.12}, we have
\begin{align*}
\|(A)\|_0&=\|L(w_m)(S_m-I)\rho_m\|_0\leq C\|(S_m-I)\rho_m\|_2\\
&\leq CC_{s*}\mu_m^{-(s^*-2)}\|\rho_m\|_{s^*}\\
&\leq C'_{s*}\mu_m^{-(s^*-2)}C_{s*}^{m+1}\mu_m^{\beta}N(s^*).
\end{align*}
By \eqref{eq:6.6+} and the \`a priori estimate \eqref{eq:5.38},
$$
\|(B)\|_0=\|\theta_m\triangle S_m\rho_m\|_0\leq C\theta_m\| S_m\rho_m\|_2\leq C\theta_m\mu_m^2\| \rho_m\|_0\leq C\theta_m\mu_m^2\| g_m\|_0.
$$

Now by combining the estimates of $(A), (B)$  and $(C)$ ,  we obtain
\begin{equation}\label{eq:6.34}
\|g_{m+1}\|_0\leq C\Big(\mu_m^{-(s^*-2-\beta)}C_{s*}^{m+1}N(s^*)+\theta_m\mu_m^2\| g_m\|_0
+\mu_m^{2(k-2)+[\frac{n}{2}]+5}\|g_m\|_0^2\Big).
\end{equation}
Since $\theta_{m}=\|g_{m}\|_{L^\infty}$ by \eqref{eq:6.10}, we need to prove two estimates in \eqref{eq:6.13} together. By same computation, using Sobolev imbedding, we have
\begin{equation}\label{eq:6.35}
\|g_{m+1}\|_{L^\infty}\leq C\mu_m^{[\frac{n}{2}]+1}[\mu_m^{-(s^*-2-\beta)}C_{s*}^{m+1}\mathcal{N}(s^*)+\|g_m\|_{L^\infty}\mu_m^{2}\| g_m\|_0
+\mu_m^{2(k-2)+[\frac{n}{2}]+5}\|g_m\|_0^2].
\end{equation}

By comparing the powers of $\mu_m$ on both sides of \eqref{eq:6.34} and \eqref{eq:6.35}, we can choose $a>0$ and
large $s^*>s$ such that
\begin{equation}\label{eq:6.21}\left\{
\begin{array}{l}
2(k-2)+2[\frac{n}{2}]+6+a \gamma \le 2 a-1\\
s^*-[\frac{n}{2}]-3-\beta\ge a \gamma+1\,.
\end{array}\right.
\end{equation}
Noticing that $\mu_{m+1}=\mu_m^\gamma$, we can change \eqref{eq:6.34} and \eqref{eq:6.35} as
\begin{align*}
C\mu_{m+1}^a\|g_{m+1}\|_0\leq &C^2\Big[\mu_m^{-(s^*-2-\beta-a \gamma)}C_{s*}^{m+1}\mathcal{N}(s^*)+\|g_m\|_{L^\infty}\mu_m^{a\gamma+2}\| g_m\|_0\\
&\qquad\qquad+\mu_m^{a\gamma+2(k-2)+[\frac{n}{2}]+5}\|g_m\|_0^2\Big],\\
C\mu_{m+1}^a\|g_{m+1}\|_{L^\infty}\leq &C^2\mu_m^{[\frac{n}{2}]+1}
\Big[\mu_m^{-(s^*-2-\beta-a\gamma)}C_{s*}^{m+1}\mathcal{N}(s^*)+\|g_m\|_{L^\infty}\mu_m^{a\gamma+2}\| g_m\|_0\\
&\qquad\qquad+\mu_m^{a\gamma+2(k-2)+[\frac{n}{2}]+5}\|g_m\|_0^2\Big],
\end{align*}
from which we obtain, by using \eqref{eq:6.21} and $\mu_m>1$,
\begin{equation}\label{eq:6.24}\left\{
\begin{array}{l}
C\mu_{m+1}^a\|g_{m+1}\|_0\leq C^2\left[\mu_m^{-1}C_{s*}^{m+1}\mathcal{N}(s^*)+\|g_m\|_{L^\infty}\mu_m^{2a-1}\| g_m\|_0
+\mu_m^{2a-1}\|g_m\|_0^2\right]\\
C\mu_{m+1}^a\|g_{m+1}\|_{L^\infty}\leq C^2\left[\mu_m^{-1}C_{s*}^{m+1}\mathcal{N}(s^*)
+\|g_m\|_{L^\infty}\mu_m^{2a-1}\| g_m\|_0
+\mu_m^{2a-1}\|g_m\|_0^2\right].
\end{array}\right.
\end{equation}
Noticing $\gamma>1$, we can choose $\sigma=\sigma(s^*)>1$ so large that $\mu_m^{-1}=\sigma^{-\gamma^m}<\frac{1}{2}$ and
\begin{equation}\label{eq:6.25}
C^2\mu_m^{-1}C_{s*}^{m+1}=C^2\sigma^{-\gamma^m}C_{s^*}^{m+1}=C^2\left(\frac{C_{s^*}}{\sigma^{\frac{1}{m+1}\gamma^m}}\right)^{m+1}\le \frac{1}{4}.
\end{equation}
Inserting such   $\sigma(s^*)$ into \eqref{eq:6.24},  we have
\begin{equation*}
\left\{
\begin{array}{l}
C\mu_{m+1}^a\|g_{m+1}\|_0\leq \frac{1}{4}\mathcal{N}(s^*)+\frac{1}{2}C^2\|g_m\|_{L^\infty}\mu_m^{2a}\| g_m\|_0
+\frac{1}{2}C^2\mu_m^{2a}\|g_m\|_0^2\\
C\mu_{m+1}^a\|g_{m+1}\|_{L^\infty}\leq \frac{1}{4}\mathcal{N}(s^*)
+\frac{1}{2}C^2\|g_m\|_{L^\infty}\mu_m^{2a}\| g_m\|_0
+\frac{1}{2}C^2\mu_m^{2a}\|g_m\|_0^2.
\end{array}\right.
\end{equation*}
Set
$$
  d_{m+1}=\max\{C\mu_{m+1}^a\|g_{m+1}\|_0,\, C\mu_{m+1}^a\|g_{m+1}\|_{L^\infty}\},
$$
we get
\begin{equation*} \left\{
\begin{array}{l}
C\mu_{m+1}^a\|g_{m+1}\|_0\leq \frac{1}{4}\mathcal{N}(s^*)+d_m^2\\
C\mu_{m+1}^a\|g_{m+1}\|_{L^\infty}\leq \frac{1}{4}\mathcal{N}(s^*)+d_m^2.
\end{array}\right.
\end{equation*}
So we obtain,
\begin{equation}\label{eq:6.27}
  d_{m+1}\leq \frac{1}{4}\mathcal{N}(s^*)+d_m^2.
\end{equation}
Since
$$
\|g_0\|_0=\|G(w_0)\|_0=\|G(0)\|_0=O(\varepsilon),\,\,\, \|g_0\|_{L^\infty}=O(\varepsilon),
$$
we choose $\varepsilon(\sigma)>0$ small such that  $\mathcal{N}(s^*),\,\, \|g_0\|_{L^\infty}$ and $\|g_0\|_0$ small.
By \eqref{eq:6.24}, \eqref{eq:6.25} and $\|g_0\|_{0}\leq \|g_0\|_{s^*}\leq \mathcal{N}(s^*)$, we have
$$
d_1\leq \frac{1}{4}\mathcal{N}(s^*)+\frac{1}{4}\|g_0\|_{0}\leq \frac{1}{2}\mathcal{N}(s^*).
$$
By induction and \eqref{eq:6.27}, we see that
$$
d_{m+1}\leq \frac{1}{2}\mathcal{N}(s^*),
$$
this completes the proof of \eqref{eq:6.13}.
\end{proof}

\begin{remark}
The  estimates in Lemma \ref{lemma4.3} and \ref{lm:du}, obtained only in the special case $f=f(y)$, are also true for the general case  $f=f(y,u)$ or $f=f(y,u,Du).$  Here we only give it as an example the estimate of $\|(C)\|_0$ in the proof of Lemma \ref{lm:du}. In fact,  from \eqref{eq:6.11}, and \`a priori estimate \eqref{eq:5.38}, it follows that  if $\| w_m\|_{[\frac{n}{2}]+3}\leq 1,$
\begin{align*}
&\|w_m+t S_m\rho_m\|_{L^\infty}\leq \|w_m\|_{L^\infty}+\| S_m\rho_m\|_{L^\infty}\\
\leq& 1+\| S_m\rho_m\|_{[\frac{n}{2}]+1}
\leq  1+C\| \rho_m\|_{[\frac{n}{2}]+1}\\
\leq & 1+C\| g_m\|_{[\frac{n}{2}]+1}
\leq  1+C(\mathcal{M}([\frac{n}{2}]+1)+\| w_m\|_{[\frac{n}{2}]+3})\leq C'.
\end{align*}
and similarly, if $\| w_m\|_{[\frac{n}{2}]+4}\leq 1,$
$$
\|\nabla w_m+t \nabla S_m\rho_m\|_{L^\infty}\leq \|\nabla w_m\|_{L^\infty}+\| \nabla S_m\rho_m\|_{L^\infty}\leq C',
$$
 so we have similar estimates
$$
\|(C)\|_0\leq C\mu_m^{2(k-2)+[\frac{n}{2}]+5}\|g_m\|_0^2.
$$
for the quadratic error of  $f=f(y,u)$ or $f=f(y,u,Du).$
\end{remark}

\noindent
\textbf{Existence of $k$-convex solution of Theorem \ref{main}}

Now we employ the Nash-Moser-H\"ormander iteration to prove the existence of solution of main theorem with $s\ge 2[\frac{n}{2}]+5$.
We shall prove by induction  that, there exists $ \varepsilon_0>0$ small such that, for any $\varepsilon\in ]0,\varepsilon_0]$
\begin{equation}\label{eq:6.29}
  \|w_m\|_{s}\leq 1,\quad\forall m \in \mathds{N}.
\end{equation}
 Since $w_0=0$, we may assume that \eqref{eq:6.29} holds for $0\le l\le m$, which, by Sobolev imbedding theorem, guarantees
  the assumption of Lemma \ref{lemma4.3} and \ref{lm:du}. Interpolation inequality and \eqref{eq:6.5+} yield, for any $0\le s\leq s^*$,
$$
\|w_{m+1}\|_s\leq \sum_{l=0}^m\|S_l\rho_l\|_s\leq C_s\sum_{l=0}^m\|\rho_l\|_s\leq
C_s\sum_{l=0}^m\|\rho_l\|_{s^*}^{\frac{s}{s^*}}\|\rho_l\|_0^{1-\frac{s}{s^*}}.
$$
By \eqref{eq:6.12}, it follows that,  for $0\le l\le m$,
$$
\|\rho_l\|_{s^*}\leq  (C'_{s^*})^l\mu^\beta_{l}\mathcal{N}(s^*)
$$
with $\mu_l=\sigma^{\gamma^l}, \sigma>1, \gamma>1, \beta=\frac{4}{\gamma-1}$, and by \eqref{eq:6.13},
$$
\|\rho_l\|_{0}\leq C\|g_l\|_{0}\leq C'\mu^{-a}_{l}\mathcal{N}(s^*),
$$
thus
$$
 \|w_{m+1}\|_s\leq
C_s\sum_{l=0}^m(C'_{s^*})^{l\frac{s}{s^*}}\mu_l^{\beta\frac{s}{s^*}-a(1-\frac{s}{s^*})}
\mathcal{N}(s^*).
$$
So that we can choose $s^*$ large enough such that
$$
\beta\frac{s}{s^*}-a(1-\frac{s}{s^*})=-\tilde a<0.
$$
We choose $\varepsilon_0=\varepsilon_0(\sigma)>0$ smaller to make $\mathcal{N}(s^*)$
small enough such that for $s\ge 2[\frac{n}{2}]+5$,
$$
 \|w_{m+1}\|_{s}\leq
C_s\sum_{l=0}^\infty(C'_{s^*})^{l\frac{s}{s^*}}\mu_l^{-\tilde a}\,\mathcal{N}(s^*)\leq 1
$$
This completes the proof of \eqref{eq:6.29}.

On one hand, by \eqref{eq:6.29} there is a subsequence of $w_m$, still
denoted by itself, such that $w_m\rightarrow w$ in weak topology of  $\mathbf{H}^s(\Omega), s\ge {2[\frac{n}{2}]+5}$ and  $w_m\rightarrow w$ in $C^{[\frac{n}{2}]+4}$. Hence
$$
g_m=-G(w_m)\rightarrow -G(w)\indent \text{in}\indent  C^{[\frac{n}{2}]+2}(\Omega).
$$
On the other hand, by using \eqref{eq:6.29}, Lemma \ref{lm:du} can be applied for all $m\in \mathds{N}$, letting $m\rightarrow \infty$ in \eqref{eq:6.13} and recalling \eqref{eq:6.10}, we see that $G(w)=0$, thus
$$
u(y)=\frac{1}{2}\sum_{i=1}^{k-1}\tau_iy_i^2+P(y)+\varepsilon^{\frac{17}2}w(\varepsilon^{-2}y)\in \mathbf{H}^{s}( \Omega)
$$
is a local solution of the Hessian equation \eqref{eq:1.1}.


\section{Strict convexity of local solution}\label{section6}

In this section, we will prove that the smooth $k$-convex local-solution obtained in Section \ref{section5} is locally strict convex under the hypothesis of Theorem \ref{main}, that is, by \eqref{eq:1.3} we need to prove that, for  $ 0<t <1,\,\, y,\, z\in \Omega,\,\, y\neq z,$
\begin{equation}\label{eq:1.3+}
\sum_{i,j=1}^n \int_0^1\int_0^1u_{ij}(x(s,\mu))d\mu ds(y_i-z_i)(y_j-z_j)>0
\end{equation}
with $x(s,\mu)=(s\mu +(1-s)t)y+(s(1-\mu)+(1-s)(1-t))z$.
Recalling from \eqref{2.abd},
$$
(u_{ij})_{1\le i, j\le n}=\mathbf{r}=\left(\sum_{l=1}^{k-1}\delta_{i}^{j}\delta_j^l\tau_l+ P_{ij}(\varepsilon^2x)+
\varepsilon^{\frac{9}{2}} w_{ij}(x)
\right),
$$
we separate this matrix into two parts: one is
$$
r_{ij}=\sum_{l=1}^{k-1}\delta_{i}^{j}\delta_j^l\tau_l+ P_{ij}(\varepsilon^2 x)+ \varepsilon^{\frac{9}{2}}
w_{ij}(x),\indent 1\le i, j\le k-1,
$$
the principal term of which is $\sum_{l=1}^{k-1}\delta_{i}^{j}\delta_j^l\tau_l$ and obviously can control
the perturbation term $ P_{ij}(\varepsilon^2 x)+ \varepsilon^5
w_{ij}(x)$ for small $\varepsilon>0$.  The other  is
$$
r_{ij}=P_{ij}(\varepsilon^2x)+ \varepsilon^{\frac{9}{2}}
w_{ij}(x),\indent i\geq k  \ \ \text{or} \ \ j\geq k,
$$
for which, in order to control the perturbation term,  our idea is to prove
\begin{equation}\label{7.1}
w_{ij}(x',0)=w_{ij p}(x',0)=0, \indent k\leq p\leq n,
\end{equation}
if $i\geq k$ or $\ j\geq k$. Then the Taylor expansion with respect to $x''=(x_{k},\cdots,x_n)$
$$
w_{ij}(x)=w_{ij}(x',x'')=w_{ij}(x',0)+\sum_{p=k}^nw_{ij p}(x',0) x_p+O(|x''|^2),
$$
yields
$$
r_{ij}(x)=P_{ij}(\varepsilon^2 x)+\varepsilon^{\frac{9}{2}}O(|x''|^2), \indent i\geq k  \ \ \text{or} \ \ j\geq k,
$$
which, together with \eqref{eq:8.14++}, implies the minor matrix $(r_{ij})_{k\leq i,j\leq n} $ is strictly diagonally dominant with
$$
\left\{
\begin{array}{l}
|r_{ij}|\leq O(1){\varepsilon^4|x''|}+ O(1)\varepsilon^{\frac{9}{2}}|x''|^2, \quad  1\leq i\leq k-1, k\leq j\leq n, \\
r_{j_0j_0}(x)=P_{j_0j_0}(\varepsilon^2x)+\varepsilon^{\frac{9}{2}} w_{j_0j_0}(x)\geq \alpha \varepsilon^4|x''|^2+\sum_{i=k,i\neq j_0}^n|r_{ij_0}|,  \indent k\leq j_0\leq n.
\end{array}\right.
$$
Choose $\varepsilon>0$ small $(\varepsilon\ll \alpha)$ enough such that, for $k\leq j\leq n$,
$$
r_{jj}(x)-\sum_{i=k,i\neq j}^n|r_{ij}(x)|\geq \alpha\, \varepsilon^4|x''|^2,
$$
Then Cauchy inequality yields
\begin{align*}
&\sum_{i,j=1}^n r_{ij}\xi_i\xi_j\\
&\geq\frac{1}{2}\sum_{i=1}^{k-1}\tau_i|\xi_i|^2+2\sum_{1\leq i\leq k-1,k\leq j\leq n}O(1)[\varepsilon^4|x''|+ \varepsilon^{\frac{9}{2}}|x''|^2]\xi_i\xi_j+\sum_{i=k}^n \alpha \varepsilon^4|x''|^2|\xi_i|^2\\
&\geq \frac{1}{4}\sum_{i=1}^{k-1}\tau_i|\xi_i|^2 +\frac{1}{2}\sum_{i=k}^n\alpha \varepsilon^4|x''|^2|\xi_i|^2, \indent \forall \xi \in \mathbb{R}^n
\end{align*}
from which,  it follows  that \eqref{eq:1.3+} holds. In fact ,  setting  $\xi=y-z$ and recalling $x''=x''(s,\mu)=(s\mu +(1-s)t)y''+(s(1-\mu)+(1-s)(1-t))z''$, we have
\begin{equation*}
\begin{split}
&\sum_{i,j=1}^n \int_0^1\int_0^1u_{ij}(x)d\mu ds(y_i-z_i)(y_j-z_j)\\
&\geq \left( \frac{1}{4}\sum_{i=1}^{k-1}\tau_i|y_i-z_i|^2+\frac 12\alpha \varepsilon^4 b(t) \sum_{i=k}^n|y_i-z_i|^2\right)>0
\end{split}
\end{equation*}
with
$$
b(t)=\int^1_0\int^1_0|[s\mu +(1-s)t]y''+[s(1-\mu)+(1-s)(1-t)]z''|^2d\mu d s>0,
$$
for any $0<t<1$ and $y''\not=z''$, this inequality is true because
$$
\Big\{(s,\mu)\in [0,1]\times[0,1]; \,\,\big[s\mu +(1-s)t\big]y''+\big[s(1-\mu)+(1-s)(1-t)\big]z''=0\Big\}
$$
lies on a hyperbolic curve in the $(s,\mu)-$ plane and then  its  Lebesgue measure for $d\mu ds$ is zero. So that $u$ is strictly convex on $\Omega$.

We prove now \eqref{7.1} by the following two lemmas. From the explicit expression of $P(y)$ in \eqref{eq:8.8} together with $u(0)=0$ and $\nabla u(0)=0$, we see that $w(0)=0$ and $\nabla w(0)=0$.  Moreover, we have

\begin{lemma}\label{th:6.5}
 Let  $u$ be a solution of equation \eqref{eq:8.8++} in the form of \eqref{eq:2.abc} with $\|w\|_{C^3}\leq 1$, then
 \begin{equation}\label{eq:8.33}
 w_{ij}(x',0)=0, \indent i\geq k  \ \ \text{or} \ \ j\geq k.
\end{equation}
\end{lemma}

\begin{proof}
By \eqref{eq:8.8},
$$
P_{ij}(\varepsilon^2x)\mid_{x''=0}=0  \ \ \ \mbox { for} \ \ \ i\neq j, \,\, i\geq k \ \mbox{or} \ j\geq k ,
$$
then
$$
r_{ij}\mid_{x''=0}=\varepsilon^{\frac{9}{2}}w_{ij}(x',0),\,\,\, i\neq j,\,\, i\geq k \ \ \ \text{or} \ \ \ j\geq k.
$$
and
$$
\sum_{i=k}^n\sum_{j=1}^{k-1}r_{ji}^2(x',0)\sigma_{k-2,j}(\lambda_1,\ldots,\lambda_{k-1})=O(1)\varepsilon^9.
$$
$$
P_{ii}(\varepsilon^2x',0)=O(1)\varepsilon^8 \ \ \ \mbox { for} \ \ \ k\leq i\leq n,
$$
thus, for $ k\leq j\leq n$,
$$
r_{jj}(x',0)=P_{jj}(\varepsilon^2 x',0)+\varepsilon^{\frac{9}{2}}w_{jj}(x',0)=\varepsilon^{8}O(1)+\varepsilon^{\frac{9}{2}}w_{jj}(x',0).
$$
Inserting the above estimates into $S_k(\mathbf{r})=\text{\ding {172}}+\text{\ding {173}}+\text{\ding {174}}$ in Lemma \ref{th:6.4} and then multiplying both sides by $\varepsilon^{-\frac{9}{2}}$, we obtain
\begin{equation*}
\sigma_{k-1}(\lambda_1,\ldots,\lambda_{k-1})\sum_{j=k}^nw_{jj}(x',0)=\varepsilon^{\frac{9}{2}}O(1).
\end{equation*}
Letting $\varepsilon \rightarrow 0^+$, we get
$$
\sum_{j=k}^nw_{jj}(x',0)=\sigma_1(w_{kk}(x',0),w_{k+1,k+1}(x',0),\ldots,w_{nn}(x',0))=0.
$$
Using now the  identity
$$
 2\sigma_2(\lambda)=[\sigma_1(\lambda)]^2-\sum_{i=1}^m\lambda_i^2, \ \ \forall  \lambda=(\lambda_1,\cdots,\lambda_m)\in \mathbb{R}^m,
$$
it follows that
\begin{equation*}
\begin{split}
  &2\sigma_2(r_{kk}(x',0),r_{k+1,k+1}(x',0),\ldots,r_{nn}(x',0))\\
  =&[\sum_{j=k}^nr_{jj}(x',0)]^2-\sum_{j=k}^n[r_{jj}(x',0)]^2\\
  =&[\varepsilon^{8}O(1)+\varepsilon^{\frac{9}{2}}\sum_{j=k}^nw_{jj}(x',0)]^2-\sum_{j=k}^n[\varepsilon^{8}O(1)+\varepsilon^{\frac{9}{2}}w_{jj}(x',0)]^2\\
  =&O(1)\varepsilon^{16}-\sum_{j=k}^n[O(1)\varepsilon^{8}+\varepsilon^{\frac{9}{2}}w_{jj}(x',0)]^2.
\end{split}
\end{equation*}
 Multiplying by $\varepsilon^{-9}$ on both sides of  $S_k(\mathbf{r})=\text{\ding {172}}+\text{\ding {173}}+\text{\ding {174}}$ in Lemma \ref{th:6.4},  taking $x''=0$ and letting $\varepsilon\rightarrow 0$, we obtain
$$
\sum_{i=k}^n\sum_{j=1}^{k-1} w_{ji}^2(x',0)=\sum_{j=k}^n[w_{jj}(x',0)]^2= \sum_{i=k}^n\sum_{s=k,s\neq i}^n w_{si}^2(x',0)=0
$$
and then \eqref{eq:8.33} is true.
\end{proof}

Using now  \eqref{eq:8.33},  the Taylor expansion of $r_{ij}=P_{ij}(\varepsilon^2x)+\varepsilon^{\frac{9}{2}}w_{ij}(x)$
for $i\geq k$ or $j\geq k$ is of the following version
$$
 r_{ij}(x)=P_{ij}(\varepsilon^2x)+\varepsilon^{\frac{9}{2}}\sum_{p=k}^nw_{ijp}(x',0)x_p+
\frac{1}{2}\varepsilon^{\frac{9}{2}}\sum_{p,q=k}^nw_{ijpq}(x',0)x_px_q+O(1)\varepsilon^{\frac{9}{2}}|x''|^3,
$$
where  $\|w\|_{C^{[\frac{n}{2}]+3}}\leq 1$ is required. Similar to the proof of  Lemma \ref{th:6.5}, we can obtain
$$
\sum_{i=k}^nw_{iip}(x',0)=\sum_{i=k}^nw_{iipp}(x',0)=0, \quad k\leq p\leq n.
$$
And also :
\begin{lemma}\label{th:6.7}
 Let  $u$ be a solution of equation \eqref{eq:8.8++} in the form of \eqref{eq:2.abc}, and $\|w\|_{C^{[\frac{n}{2}]+4}}\leq 1 $, then
$$
w_{ijp}(x',0)=0, \indent i\geq k  \ \ \text{or} \ \ j\geq k,\,\, k\leq p\leq n.
$$
\end{lemma}


\section {Appendix}
\label {section7}

In this section, we will estimate the eigenvalues and eigenvectors of the following Hessian matrix
$$
\mathbf{r}=(r_{ij})_{1\le i, j\le n}=\left(\sum_{l=1}^{k-1}\delta_{i}^{j}\delta_j^l\tau_l+\varepsilon \tilde w_{ij}(x)\right)_{1\leq i, j\leq n}
$$
which is a small perturbation of diagonal matrix $\left(\sum_{l=1}^{k-1}\delta_{i}^{j}\delta_j^l\tau_l\right)_{1\leq i, j\leq n}$.
 For any function $\tilde w\in C^2$, we can find  out an orthogonal matrix
$T(x,\varepsilon)$ satisfying
\begin{equation*}
 T(x,\varepsilon)\,\mathbf{r}\,\,\Prefix^{t}T(x,\varepsilon)=\textup{diag}\left[\lambda_{1}(x,\varepsilon),
\lambda_{2}(x,\varepsilon),\ldots,\lambda_{n}(x,\varepsilon)\right],
\end{equation*}
Let
\begin{gather*}
T(x,\varepsilon)=
\begin{pmatrix}
T_1\\
T_2\\
\vdots\\
T_n
\end{pmatrix},\ \ \ T(x,\varepsilon)\mid_{\varepsilon=0}=
Id
\end{gather*}
with $T_i=(T_{i1},T_{i2},\ldots,T_{in})$ the corresponding unit eigenvectors of $\lambda_i, i=1,2,\ldots,n$.
In \cite{HZ}, there also exists such an orthogonal matrix
$T(x,\varepsilon)$, the  estimates of good regularity  for all of its entries can be easily
obtained because all of its eigenvalues are different from each
other. Now we can only give the estimates
of $T_{i}$ for $1\leq i \leq k-1$, because in our case
$\lambda_j (x,\varepsilon)$ ($ k\leq j\leq n$)  are around
zero and there is no distinct gap among them, they are not necessarily smooth in $x$ and
$\varepsilon,$ so are the corresponding eigenvectors $T_i (x, \varepsilon) \, (k\leq i\leq n).$
 But the following estimates are enough for us.

\begin{proposition}\label{lm:matrix}
Suppose that $\tilde w$ is smooth and $\|\tilde w\|_{C^{3}}\leq 1$.
Then, $T_i(x,\varepsilon), \ \ 1\leq i\leq k-1$, is smooth in $(x,\varepsilon)\in
\overline\Omega\times[0,\varepsilon_{0}]$ for some positive
$0<\varepsilon_{0}\ll 1$, and
\begin{align}\label{eq:4.6}
&\sum_{i=1}^{k-1}|\lambda_{i}-\tau_i|\leq C\varepsilon \sum_{i,j=1}^n|\tilde w_{ij}(x)|,\indent\sum_{i=k}^n|\lambda_{i}|\leq C\varepsilon^{\frac{1}{2}}(\sum_{i,j=1}^n|\tilde w_{ij}(x)|)^{\frac{1}{2}},\\
\label{eq:4.7}
&\sum_{i=1}^{k-1}|T_{ii}(x,\varepsilon)-1|
+\sum_{i=1}^{k-1}\sum_{j=1,j\neq i}^n|T_{ij}(x,\varepsilon)|\leq C\varepsilon\sum_{i,j=1}^n|\tilde w_{ij}(x)|,\\
\label{eq:4.41}
&\sum_{i=1}^{k-1}\sum_{j=1}^n |D T_{ij}(x,\varepsilon)|\leq C\varepsilon\sum_{i,j,l=1}^n|\tilde w_{ijl}(x)|
\end{align}
with $C$ independent of $\varepsilon$ and $\tilde w$.

Moreover, if  $\tilde w(x)$ is periodic in $x_i \ \ (1\leq i\leq k-1)$, then so is each  of  $T_{ij}$,
while  $T_{ij}$ has the same regularity as $D^2\tilde w$ and is $C^\infty$ in $\varepsilon$ for $1\leq i\leq k-1, 1\leq j\leq n$.
\end{proposition}

\begin{proof} Noting
$$
F(t)=\det (\mathbf{r}-t \, \mathbf{I}).
$$
Now we denote $R(x,\varepsilon)$, $R_i(x,\varepsilon)$ and $R_{ij}(x,\varepsilon)$ as the different
functions, which are smooth in $x$ , $\varepsilon$ and $\tilde w,$  with
the properties
\begin{equation}\label{eq:4.39}
  |R(x,\varepsilon)|+ |R_i(x,\varepsilon)|+|R_{ij}(x,\varepsilon)|\leq C\varepsilon \sum_{i,j=1}^n|\tilde w_{ij}(x)|
\end{equation}
and $C$ being independent of $x$ and $\varepsilon$. Firstly, using the condition $\|\tilde w\|_{C^{2}}\leq 1$ and
$0<\varepsilon_{0}\ll 1$, we have
$$
F(t)=\prod_{i=1}^{k-1}(\tau_i-t)t^{n-k+1}+R(x,\varepsilon).
$$
Noticing $\tau_1>\tau_2>\ldots>\tau_{k-1}>0$, we take $\delta=\frac{1}{4}\min_{1\leq i\leq k-2}(1-\frac{\tau_{i+1}}{\tau_i},\frac{\tau_{i}}{\tau_{i+1}}-1)>0$ if $k>2$ and $\delta=\frac{1}{4}$ if $k=2$.
For fixed $i_0$ with $1\leq i_0\leq k-1$,
we have
$$
F((1-\delta)\tau_{i_0})=\prod_{i=1}^{i_0-1}(\tau_i-(1-\delta)\tau_{i_0})\delta\tau_{i_0}
\prod_{i=i_0+1}^{k-1}(\tau_i-(1-\delta)\tau_{i_0})
((1-\delta)\tau_{i_0})^{n-k+1}+R_1(x,\varepsilon)
$$
and
$$
F((1+\delta)\tau_{i_0})=-\prod_{i=1}^{i_0-1}(\tau_i-(1+\delta)\tau_{i_0})
\delta\tau_{i_0}
\prod_{i=i_0+1}^{k-1}(\tau_i-(1+\delta)\tau_{i_0})
((1+\delta)\tau_{i_0})^{n-k+1}+R_2(x,\varepsilon).
$$
By the choice of $\delta$, we have
$$
\tau_i-(1+\delta)\tau_{i_0}>\delta\tau_{i_0-1},\indent 1\leq i\leq i_0-1
$$
and
$$
\tau_i-(1-\delta)\tau_{i_0}<-\delta\tau_{i_0+1},\indent i_0+1\leq i\leq k-1,
$$
therefore, when $0<\varepsilon\ll \delta\tau_{k-1},$ we obtain
$$
F((1-\delta)\tau_{i_0})F((1+\delta)\tau_{i_0})<0
$$
and, by virtue of intermediate value theorem, there exists an
eigenvalue, denoted by $\lambda_{i_0},$ such that
$$
(1-\delta)\tau_{i_0}<\lambda_{i_0}<(1+\delta)\tau_{i_0},\indent F(\lambda_{i_0})=0,
$$
which yields
$$
\tau_i-\lambda_{i_0}\geq\tau_i-(1+\delta)\tau_{i_0}>\delta\tau_{k-1},\indent 1\leq i\leq i_0-1
$$
and
$$
\tau_{i}-\lambda_{i_0}\leq\tau_i-(1-\delta)\tau_{i_0}<-\delta\tau_{k-1},\indent i_0+1\leq i\leq k-1.
$$
From $0=F(\lambda_{i_0})=\prod_{i=1}^{i_0-1}(\tau_i-\lambda_{i_0})
(\tau_{i_0}-\lambda_{i_0})\prod_{i=i_0+1}^{k-1}(\tau_i-\lambda_{i_0})
(\lambda_{i_0})^{n-k+1}+R_3(x,\varepsilon),$
it follows that
$$
\tau_{i_0}-\lambda_{i_0}=-\frac{R_3(x,\varepsilon)}
{\prod_{i=1}^{i_0-1}(\tau_i-\lambda_{i_0})\prod_{i=i_0+1}^{k-1}(\tau_i-\lambda_{i_0})
(\lambda_{i_0})^{n-k+1}}
$$
and then
$$
\lambda_{i_0}=\tau_{i_0}+R_4(x,\varepsilon).
$$
Now we have proved that there are eigenvalues $(\lambda_i)_{i=1}^{k-1}$ such that
$$
\lambda_i=\tau_i+R(x,\varepsilon), \indent 1\leq i\leq k-1.
$$
Let us  express all the eigenvalues as $(\lambda_i)_{i=1}^{n}.$ Then, since $\|\tilde{w}\|_{C^3}\leq 1,$
$$
S_1(r)=\sum_{i=1}^{k-1}\tau_i +R(x,\varepsilon),\,\,\,
S_2(r)=\sigma_2(\tau_1,\ldots,\tau_{k-1})+R(x,\varepsilon).
$$
Since $S_k(r)$ is invariant under orthogonal transformation, then
$$
\sum_{i=1}^{k-1}\tau_i+R(x,\varepsilon)=
S_1(\mathbf{r})=\sigma_1(\lambda)=\sum_{i=1}^{k-1}\lambda_i+\sum_{i=k}^n\lambda_i
$$
and
$$
\sigma_2(\tau_1,\ldots,\tau_{k-1})+R(x,\varepsilon)=S_2(\mathbf{r})=\sigma_2(\lambda).
$$
from which, using
$$
2\sigma_2(\lambda)=\sigma_1(\lambda)^2-\sum_{i=1}^n\lambda_i^2.
$$
we see that
\begin{equation*}
   \sigma_2(\lambda)
   = \sigma_2(\tau_1,\ldots,\tau_{k-1})-\frac{1}{2}\sum_{i=k}^n\lambda_i^2+R(x,\varepsilon)
 \end{equation*}
and then
$$
\sum_{i=k}^n\lambda_i^2=R(x,\varepsilon),
$$
which implies
\begin{equation}\label{eq:4.8}
|\lambda_i|\leq C\varepsilon^{\frac{1}{2}}(\sum_{i,j=1}^n|\tilde w_{ij}(x)|)^{\frac{1}{2}},\indent k\leq i\leq n.
\end{equation}
This completes the proof of \eqref{eq:4.6}.

Now we pass to prove \eqref{eq:4.7}. Let $T_1$ be the eigenvector
corresponding to $\lambda_1,$ then $T_1$ satisfies the linear
equation $\mathbf{r}T_1-\lambda_1T_1=0$ and $\mbox{rank}(\mathbf{r}-\lambda_1 \mathbf{I})\leq n-1.$
To solve $T_1,$ we will use the Gaussian elimination procedure. Noticing $\lambda_i=\tau_i+R(x,\varepsilon)$
 for $1\leq i\leq k-1$, we can write the matrix $\mathbf{r}-\lambda_1 \mathbf{I}$ as
 \begin{gather*}
\begin{pmatrix}
R_{11}&R_{12}&\ldots&R_{1,k-1}&R_{1k}&\ldots& R_{1n}\\
R_{21}&R_{22}+\tau_2-\tau_1&\ldots&R_{2,k-1}&R_{2,k}&\ldots&R_{2,n}\\
\vdots&\vdots&\vdots&\vdots&\vdots&\vdots\\
R_{k-1,1}& R_{k-1,2}&\ldots& R_{k-1,k-1}+\tau_{k-1}-\tau_1&R_{k-1,k}&\ldots&R_{k-1,n}\\
R_{k,1}& R_{k,2}&\ldots& R_{k,k-1}&
R_{k,k}-\tau_1&\ldots&R_{k,n}\\
\vdots&\vdots&\vdots&\vdots&\vdots&\vdots\\
R_{n,1}&R_{n,2}&\ldots&R_{n,k-1}&R_{n,k}&\ldots&R_{n,n}-\tau_1.
\end{pmatrix}
\end{gather*}
Since each $R_{jj}$, $2\leq j\leq k-1$, is dominated by $\tau_j-\tau_1$ and each $R_{ii}$,
$k\leq i\leq n$, is dominated by $\tau_1$, then $\mbox{rank}(\mathbf{r}-\lambda_1 \mathbf{I})=n-1$.
 Taking the (row) elementary operations repeatedly, we can transform $\mathbf{r}-\lambda_1 \mathbf{I}$ and still denote it as
\begin{gather*}
\begin{pmatrix}
R_{11}&0&\ldots&0&0&\ldots& 0\\
R_{21}&R_{22}+\tau_2-\tau_1&\ldots&0&0&\ldots&0\\
\vdots&\vdots&\vdots&\vdots&\vdots&\vdots\\
R_{k-1,1}&0&\ldots&R_{k-1,k-1}+\tau_{k-1}-\tau_1&0&\ldots&0\\
R_{k,1}& 0&\ldots& 0&
R_{k,k}-\tau_1&\ldots&0\\
\vdots&\vdots&\vdots&\vdots&\vdots&\vdots\\
R_{n,1}&0&\ldots&0&0&\ldots&R_{n,n}-\tau_1.
\end{pmatrix}
\end{gather*}

Since $\mbox{rank}(\mathbf{r}-\lambda_1 \mathbf{I})=n-1$, we see that $R_{11}\equiv 0$ and the solutions of the equation  $\mathbf{r}T_1-\lambda_1T_1=0$  is in the form
\begin{gather}\label{eq:4.42}
\begin{pmatrix}
T_{11}\\
T_{12}\\
\vdots\\
T_{1,k-1}\\
T_{1,k}\\
\vdots\\
T_{1n}\end{pmatrix}=T_{11}\begin{pmatrix}
1\\
\frac{R_{21}(x,\varepsilon)}{\tau_1-\tau_2-R_{22}(x,\varepsilon)}\\
\vdots\\
\frac{R_{k-1,1}(x,\varepsilon)}{\tau_1-\tau_{k-1}-R_{k-1,k-1}(x,\varepsilon)}\\
\frac{R_{k,1}(x,\varepsilon)}{\tau_1-R_{k-1,k-1}(x,\varepsilon)}\\
\vdots\\
\frac{R_{n,1}(x,\varepsilon)}{\tau_1-R_{nn}(x,\varepsilon)}\end{pmatrix},\ \  \mbox{with} \ \  T_{11}\neq 0,
\end{gather}
and each $R_{ij}(x,\varepsilon)$ has the property \eqref{eq:4.39}.
 Because we need  $T_1(x,\varepsilon)\mid_{\varepsilon=0}=(1,0,\ldots,0)$, we can choose suitable
$T_{11}>0$ such that $\|T_1\|=1,$ then $T_1=(T_{11},T_{12},\ldots,T_{1n})$
satisfies \eqref{eq:4.7}.  The vectors $T_i$ $ (2\leq i \leq k-1)$ can be obtained by
the same way as $T_1.$ Also it follows from \eqref{eq:4.42} that \eqref{eq:4.7} and \eqref{eq:4.41} hold.

Because each entry of the matrix $\mathbf{r}=(r_{ij})$ is periodic in $x_1,\ldots,x_{k-1}$ and the elementary operations above do not change the periodicity, then each $T_i (1\leq i \leq k-1)$ is also periodic. Also by \eqref{eq:4.42}, when $1\leq i\leq k-1,$ $T_{ij}$ has the same regularity as $D^2\tilde w$ and is $C^\infty$ in $\varepsilon.$
\end{proof}

\remark {Since $\lambda_i, \  k\leq i\leq n,$  are all around 0, then
\eqref{eq:4.8} shows that they are not necessarily smooth in $x$ and
$\varepsilon,$ so are the corresponding eigenvectors $T_i $. }

\bigskip
\noindent{\bf Acknowledgements.}  The research of first author is
supported by the National Science Foundation of China No.11171339
and Partially supported by National Center for Mathematics and
Interdisciplinary Sciences. The research of the second author was
supported partially by``The Fundamental Research Funds for Central
Universities of China".


\begin{thebibliography}{\small}
\bibitem{AG}
     Alinhac, S.,  G\'{e}rard,  P. (1991). {\em Op\'{e}rateurs pseudo-diff\'{e}rentiels et th\'{e}or\`{e}me  de Nash-Moser.} EDP Sciences\ CNRS \'{E}DITIONS, Paris, France.

\bibitem{BG}
   Bian, B.,  Guan, P. (2009). A microscopic convexity principle for nonlinear partial
differential equations. {\em Invent math.} {\bf 177}: 307-335.

\bibitem{CH}
 Chen,T.,  Han, Q. (2016). Smooth local solutions to Weingarten equations and $k$-equations.
{\em Discrete Contin. Dyn. Syst.} {\bf 36}: 653-660.

 \bibitem{GB}
  Guan, B. (1999). The Dirichlet problem for Hessian equations on Riemannian manifolds. {\em Calculus of Variations and Partial Differential Equations.} {\bf 8(1)}: 45-69.

\bibitem{GG}
Guan, B.,  Guan, P. (2002). Convex hypersurfaces of prescribed curvatures. {\em Annals of Mathematics},{\bf 156(2)}:  655-673.


\bibitem{GM}
 Guan,P.,  Ma, X. N. (2003).  The Christoffel-Minkowski problem I: Convexity of solutions of a Hessian equation. {\em Invent. Math.}
{\bf 151}: 553-577.

\bibitem{GLM}
 Guan, P.,  Lin C. S., Ma and X. N. (2006). The Christoffel-Minkowski problem II: Weingarten curvature equations. {\em Chin. Ann. Math.,
Ser. B} {\bf 27}:  595-614.

\bibitem{GMZ}
 Guan, P. , Ma, X. N.,  Zhou, F. (2006). The Christoffel-Minkowski problem III: existence and convexity of admissible solutions. {\em Comm. Pure Appl. Math.} {\bf 59}: 1352-1376.

 \bibitem{GSX}
  Guan B, Spruck J, Xiao L. (2014). Interior curvature estimates and the asymptotic plateau problem in hyperbolic space. {\em  Journal of Differential Geometry.} {\bf 96(2)}: :201-222.

\bibitem{Har}
      Hartman, P. (1965). On isometric immersions in Euclidean space of manifolds with non-negative sectional curvatures. {\em Trans. Amer. Math. Soc.}, {\bf 115}: 94-94.

\bibitem{HH}
     Han, Q., Hong, J. X. (2006). {\em Isometric embeddings of Riemannian manifolds in Euclidean spaces.}  Mathematical Surveys and Monographs,Volume 130,
     American Mathematical Society, Providence, RI.

\bibitem{HL}
     Harvey, R.,  Lawson,  H. B. (1982). Calibrated geometries. {\em Acta Math.} {\bf 148}:  47-157.

\bibitem{HZ}
      Hong, J. , Zuily, J. (1987). Exitence of $C^{\infty}$ local solutions for the Monge-Amp\'{e}re equation,  {\em Invent. Math.} {\bf 89}: 645-661.

\bibitem{ITW}
 Ivochkina, N. M.,  Trudinger,  N. S., Wang,  X. J. (2004). The Dirichlet  problem for
    degenerate Hessian equations, {\em Comm.Part. Diff. Equat.} {\bf 29}:  219-235.


\bibitem{IPY}
 Ivochkina, N. M.,  Prokofeva, S. I.,  Yakunina,  G. V.  (2012). The G{\aa}rding cones in the modern theory of fully nonlinear  second order differential equations,
      {\em J. Math. Sci.} {\bf 184}:   295-315.

\bibitem{LMX}
 Liu, P.,  Ma,  X.-N.,  Xu,  L.  (2010). A Brunn-Minkowski inequality for the Hessian eigenvalue in three-dimensional convex domain, {\em Adv. Math.} {\bf 225}: 1616-1633.

\bibitem{MX}
  Ma, X.-N.,    Xu, L.  (2008). The convexity of solution of a class Hessian equation in bounded convex domain in $\mathbf{R}^3$,   {\em  J. Funct. Anal.}  {\bf 255}: 1713-1723.

\bibitem{SP}
 Paolo, S.  (2012). Convexity of solutions and Brunn-Minkowski inequalities for Hessian equations in $\mathbb{R}^3$,  {\em Adv. Math.}  {\bf 229}:  1924-1948.

\bibitem{Shi}
 Shiohama, K. (1967).  Cylinders in Euclidean space $E^{2+n}$.  {\em Kodai Mathematical Seminar Reports 1967. }

\bibitem{TWX}
 Tian, G., Wang,  Q.,  Xu,  C. J. (2016).  $C^{\infty}$ local solutions of elliptical $2$-Hessian equation
in $\mathbb{R}^{3}$. {\em Discrete Contin. Dyn. Syst.} {\bf 36} : 1023-1039.

\bibitem{TWX2}
Tian, G., Wang,  Q.,  Xu,  C. J. (2016). Local solvability of the $k$-Hessian equations, {\em Sci. China Math.}
{\bf 59} : 1753-1768.

\bibitem{Ush}
 Ushakov, V. (1999).  Developable surfaces in Euclidean space. {\em J. Austr. Math. Soc.} {\bf 66} :388-402.

\bibitem{W} Wang,  X.  J. (2009). {\em The $k$-Hessian equation.}  Lecture Notes in  Math., Springer, Dordrecht .

\end{thebibliography}
\end{document}